\documentclass[12pt,reqno]{article}
mm \textheight=8.9in

\usepackage{amssymb}
\usepackage{amsmath}
\usepackage{amsthm}
\usepackage{pb-diagram}
\usepackage{color}
\newcommand{\A}{{\mathcal{A}}}

\newcommand{\br}[3]{{$#1$}$\lower4pt\hbox{$\tp\atop\raise4pt \hbox{$\scriptscriptstyle{#2}$}$} ${$#3$}}
\newcommand{\tw}[3]{{$#1$}${\,\scriptscriptstyle {#2}}\atop\raise9pt\hbox{$\scriptstyle\tp$} ${$#3$}}
\newcommand{\ttps}[2]{{#1}\raise5pt\hbox{$\lower12pt\hbox{$\scriptstyle\tp$}\atop \lower0pt\hbox{$\tilde\;$}$}\raise4.5pt\hbox{${\scriptstyle{#2}}$}}
\newcommand{\st}[1]{\mbox{${\,\scriptscriptstyle {#1}}\atop\raise5.5pt\hbox{$*$}$}}

\newcommand{\rd}[1]{\mbox{${\,\scriptscriptstyle {#1}}\atop\raise5.5pt\hbox{$\bullet$}$}}
\newcommand{\rt}[1]{\otimes_\chi}
\newcommand{\lt}[1]{\mbox{${\,\scriptscriptstyle {#1}}\atop\raise5.5pt\hbox{$\ltimes$}$}}
\newcommand{\btr}{\raise1.2pt\hbox{$\scriptstyle\blacktriangleright$}\hspace{2pt}}
\newcommand{\btl}{\raise1.2pt\hbox{$\scriptstyle\blacktriangleleft$}\hspace{2pt}}

\newcommand{\lcr}{\raise1.0pt \hbox{${\scriptstyle\rightharpoonup}$}}
\newcommand{\rcr}{\raise1.0pt \hbox{${\scriptstyle\leftharpoonup}$}}

\newcommand{\ttp}{{\lower12pt\hbox{$\tp$}\atop \hbox{$\tilde\;$}}}

\newcommand{\id}{\mathrm{id}}

\newcommand{\f}{\mathfrak{f}}

\newcommand{\Ru}{\mathcal{R}}

\newcommand{\E}{\mathcal{E}}

\newcommand{\Mc}{\mathcal{M}}

\newcommand{\Q}{\mathcal{Q}}
\renewcommand{\O}{\mathcal{O}}

\newcommand{\C}{\mathbb{C}}
\newcommand{\Z}{\mathbb{Z}}

\newcommand{\N}{\mathbb{N}}

\newcommand{\tp}{\otimes}
\newcommand{\vt}{\vartheta}

\newcommand{\U}{U}

\newcommand{\ve}{\varepsilon}
\newcommand{\gm}{\gamma}
\newcommand{\dt}{\delta}

\newcommand{\op}{\oplus}
\newcommand{\la}{\lambda}

\newcommand{\End}{\mathrm{End}}

\newcommand{\Span}{\mathrm{Span}}

\newcommand{\Tr}{\mathrm{Tr}}

\newcommand{\Rm}{\mathrm{R}}

\newcommand{\ad}{\mathrm{ad}}

\newcommand{\La}{\Lambda}

\newcommand{\g}{\mathfrak{g}}
\renewcommand{\b}{\mathfrak{b}}
\renewcommand{\k}{\mathfrak{k}}
\newcommand{\h}{\mathfrak{h}}
\newcommand{\mub}{\boldsymbol{\mu}}

\newcommand{\nb}{\boldsymbol{n}}

\newcommand{\s}{\mathfrak{s}}
\renewcommand{\o}{\mathfrak{o}}
\newcommand{\n}{\mathfrak{n}}

\newcommand{\m}{\mathfrak{m}}

\newcommand{\nn}{\nonumber}
\newcommand{\p}{\mathfrak{p}}
\renewcommand{\l}{\mathfrak{l}}
\renewcommand{\c}{\mathfrak{c}}

\newcommand{\al}{\alpha}

\newcommand{\bt}{\beta}

\newcommand{\be}{\begin{eqnarray}}
\newcommand{\ee}{\end{eqnarray}}

\newtheorem{thm}{Theorem}[section]
\newtheorem{propn}[thm]{Proposition}
\newtheorem{lemma}[thm]{Lemma}
\newtheorem{corollary}[thm]{Corollary}

\newtheorem{remark}[thm]{Remark}
\newtheorem{definition}[thm]{Definition}

\newcount\prg

\newcommand{\parag}{\advance\prg by1 {\noindent\bf\thesection.\the\prg\hspace{6pt}}}

\begin{document}
\title{Non-Levi closed conjugacy classes of  $SO_q(N)$}
\author{
Andrey Mudrov \vspace{20pt}\\
\small Department of Mathematics,\\ \small University of Leicester, \\
\small University Road,
LE1 7RH Leicester, UK\\
\small e-mail: am405@le.ac.uk\\
%[0.3in]\select{DRAFT}
}

\date{}
\maketitle

\begin{abstract}
We construct explicit quantization of semisimple  conjugacy classes of the complex orthogonal group
$SO(N)$ with non-Levi isotropy subgroups through an operator realization
on  highest weight modules of the quantum group $U_q\bigl(\s\o(N)\bigr)$.
\end{abstract}


\section{Introduction}
This is a continuation of our recent works \cite{M3,M5} on quantization of closed conjugacy classes
of simple complex algebraic groups with non-Levi stabilizers.
There, we extended the methods that had been developed  in \cite{M2} for classes with Levi isotropy subgroups
to the non-Levi case, with the focus on the symplectic groups. In this paper, we apply those ideas  to the orthogonal
groups. This solves the quantization  problem for all non-Levi conjugacy classes of simple complex matrix groups.
Along with  \cite{M3}--\cite{DM1}, this result yields quantization of all semisimple classes of
symplectic and special linear groups and "almost all" classes of orthogonal groups.
A "thin" family of orthogonal classes with Levi stabilizer is left beyond our scope. This family
can be called "borderline" as it shares some properties of non-Levi classes. We  give
a special consideration to this case in a separate publication, \cite{AM}. Then the quantization problem will be closed for all semisimple conjugacy
classes of the four classical series of simple groups.

Observe that semisimple conjugacy classes of simple complex groups fall into two families distinguished by the type of
their isotropy subgroup: whether it is  Levi or not.
With regard to the classical matrix series, the second type appears only in the symplectic and orthogonal cases.
In the orthogonal case, the isotropy subgroup of a given class is not Levi if and only if  the eigenvalues $\pm 1$ are both present in the spectrum,
the multiplicity of $-1$ is at least $4$ and,  for even $N$, the multiplicity of $1$ is at least $4$ too.
This is what we assume in this paper.

If both eigenvalues $\pm1$ are present but the multiplicity of $-1$ (or $+1$) is $2$, the stabilizer is of Levi type:
the block $SO(2)$ rotating this eigenspace is isomorphic to $GL(1)$. Quantization
of these classes  methodologically  lies in between of  \cite{M2} and  the present work. We postpone this case
to a separate study, in order to simplify the current presentation.

Recall that closed conjugacy classes are affine subvarieties of the algebraic group $G$ of complex orthogonal $N\times N$-matrices,
\cite{St}.
We consider them as Poisson homogeneous spaces over the Poisson group $G$ equipped with the
Drinfeld-Sklyanin bracket. Their Poisson structure restricts from a Poisson structure on $G$, which
is different from the Drinfeld-Sklyanin bracket. The group itself and the conjugacy classes are
Poisson manifolds over the Poisson group $G$ with respect to the adjoint action. We are searching for quantization of the
affine coordinate ring of a class as a quotient of the quantized algebra, $\C_\hbar[G]$, of polynomial functions on $G$.

The algebra  $\C_\hbar[G]$ should not be confused with the restricted dual to the quantized universal enveloping algebra.
As  above said, they are quantizations of different Poisson structures on $G$. Rather, the algebra  $\C_\hbar[G]$
is closer to $U_q(\g)$ than to its dual and can be realized as a subalgebra in $U_q(\g)$. Therefore, $\C_\hbar[G]$ is
represented on all $U_q(\g)$-modules.
We find a $U_q(\g)$-module of highest weight such that the quotient of $\C_\hbar[G]$ by the annihilator
is a deformation of the polynomial ring of a non-Levi conjugacy class.
Contrary to the Levi case, it is not a parabolic Verma module.

The key step of our approach is finding an appropriate submodule in an auxiliary parabolic Verma module $\hat M_\la$
and pass to the quotient module $M_\la$, to realize the quantized coordinate ring of the class $G/K$
by linear operators from $\End(M_\la)$.
The module $\hat M_\la$ is associated with the quantum universal enveloping subalgebra $U_q(\l)\subset U_q(\g)$ of
a certain auxiliary Levi subgroup $L\subset G$. The subgroup $L$ is maximal among those contained in the stabilizer $K$.
We obtain $L$ by reducing the orthogonal block $SO(2m)\subset K$, which rotates the eigenspace of $-1$, to  $GL(m)\subset SO(2m)$.

Having constructed $M_\la$ we proceed to the study of the $U_q(\g)$-module $\C^N\tp M_\la$.
We find the spectrum and the minimal polynomial of the image of an invariant matrix $\Q\in \End(\C^N)\tp U_q(\g)$ in
$\End(\C^N\tp M_\la)$ whose entries generate the algebra $\C_\hbar[G/K]$.
We start from the  minimal polynomial on $\C^N\tp \hat M_\la$, which is known from \cite{M2}.
Further we analyze the  structure of $\C^N\tp \hat M_\la$ and show that a submodule responsible for a simple divisor
becomes invertible under the projection $\C^N\tp \hat M_\la \to\C^N\tp M_\la$ and drops from the minimal polynomial
of $\Q$.
This reduction yields a polynomial identity on $\Q$ which determines the conjugacy class in the classical limit.

As a result, we obtain an explicit expression of the annihilator of $M_\la$ in $\C_\hbar[G]$ in terms of the
"quantum coordinate matrix" $\Q$. This annihilator is the quantized defining ideal of the class $G/K$.
This way we obtain an explicit description of $\C_\hbar[G/K]$ as a quotient
of $\C_\hbar[G]$, in terms of generators and relations.

Non-Levi conjugacy classes include symmetric spaces $SO(N)/SO(2m)\times SO(N-2m)$.
Their quantum counterparts were studied in connection with in integrable models, \cite{FZ}, and representation theory, \cite{K}--\cite{L3}.  Contrary to our approach, quantum symmetric spaces were
viewed as subalgebras in the Hopf dual to $U_\hbar(\g)$ annihilated by certain
coideal subalgebras, the quantum stabilizers, \cite{NS}--\cite{L4}.
That is possible for symmetric classes since they admit "classical points", where the Poisson bracket turns zero.
At the quantum level, classical points give rise to one-dimensional representations of $\C_\hbar[G]$. Other conjugacy classes do not
admit classical points, so our method of quantization remains the most general.
\section{Classical conjugacy classes}
\label{SecCCC}
Throughout the paper, $G$ designates the algebraic group $SO(N)$, $N\geqslant 7$ or $N=5$, of orthogonal matrices preserving
a non-degenerate symmetric bilinear form $(C_{ij})_{i,j=1}^{N}$ on the complex vector space $\C^{N}$; the
Lie algebra of $G$ will be denoted by $\g$.
We choose the realization  $C_{ij}=\delta_{ij'}$, where
$\dt_{ij}$ is the Kronecker symbol, and
$i'=N+1-i$ for $i=1,\ldots, N$.

The polynomial ring $\C[G]$ is generated by the matrix coordinate functions $(A_{ij})_{i,j=1}^{N}$,
modulo the set of $N^2$ relations written in the matrix form as
\be
ACA^t=C.
\label{ideal_group}
\ee
Strictly speaking, this equation defines the group $O(N)$, but we will ignore this distinction, because
the relation $\det A=1$ will be automatically covered by the defining relations of conjugacy classes.

The right conjugacy action of $G$ on itself induces a left  action on $\C[G]$ by duality; the matrix $A$  is invariant
as an element of $\End(\C^{N})\tp \C[G]$.

The group $G$ is equipped with the Drinfeld-Sklyanin bivector field
\be
\{A_1,A_2\}=\frac{1}{2}(A_2A_1r-r A_1A_2),
\label{poisson_br_DS}
\ee
where $r\in \g\tp \g$ is a solution of the classical Yang-Baxter equation, \cite{D}.
This equation is understood in $\End(\C^{N})\tp \End(\C^{N})\tp \C[G]$, and the
subscripts indicate the natural tensor factor embeddings of $\End(\C^{N})$ in $\End(\C^{N})\tp \End(\C^{N})$, as usual in the literature.

The bivector field (\ref{poisson_br_DS}) is skew-symmetric when restricted to  $G$ and
defines a Poisson  bracket making $G$ a Poisson group.
We fix the standard solution of the classical Yang-Baxter equation:
\be
r=\sum_{i=1}^{N}(e_{ii}\tp e_{ii}-e_{ii}\tp e_{i'i'})
+2\sum_{i,j=1\atop i>j}^{N}(e_{ij}\tp e_{ji}-e_{ij}\tp e_{i'j'}).
\ee
At the end of the article, we lift this restriction to include an arbitrary factorizable
r-matrix, \cite{BD}. This extends our results to arbitrary quasitriangular quantum orthogonal groups.

We regard the group $G$ as a $G$-space under the  conjugation action.
The object of our study is another Poisson structure on $G$,
\be
\{A_1,A_2\}=\frac{1}{2}(A_2r_{21}A_1-A_1rA_2+A_2A_1r-r_{21}A_1A_2),
\label{poisson_br_sts}
\ee
see \cite{STS}.
It is compatible with the conjugation action and makes $G$ a Poisson space
over the Poisson group $G$ equipped with the Drinfeld-Sklyanin bracket (\ref{poisson_br_DS}).

We reserve  $n$ to denote the rank of the Lie algebra $\s\o(N)$, so $N$ is either $2n$ or $2n+1$.
A semisimple conjugacy class $O\subset G$ consists of diagonalizable matrices and is
determined by the multi-set of eigenvalues
$S_O=\{\mu_i,\mu_i^{-1}\}_{i=1}^n\cup \{1\}$, where $\{1\}$ is present when $N$ is odd.
Every eigenvalue $\mu$ enters $S_O$ with its reciprocal $\mu^{-1}$ and, in particular, may degenerate to $\mu=\mu^{-1}=\pm 1$.
For a class to be non-Levi,  both $+1$ and $-1$ should be in  $S_O$. Moreover, the multiplicity of $-1$ is assumed to be $4$ or higher
as well as the multiplicity of $+1$ for even $N$.

In terms of Dynkin diagram, a  Levi subgroup is obtained by scraping out
a subset of nodes, while non-Levi isotropy subgroups are obtained from the affine Dynkin diagrams:
$$
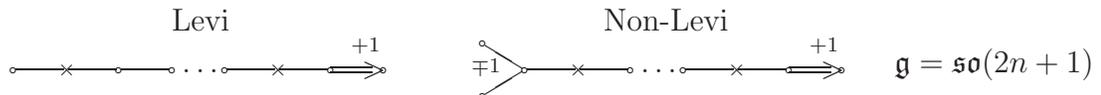
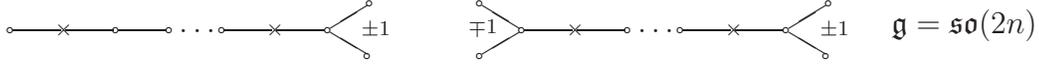
\begin{picture}(150,40)
\put(60,25){Levi} \put(128,17){$\scriptstyle{ +1}$}
  \put(0,10){\circle{2}}  \put(1,10){\line(1,0){18}}
  \put(17,8){$\scriptstyle{\times}$}  \put(41,10){\line(1,0){18}}
  \put(40,10){\circle{2}}  \put(21,10){\line(1,0){18}}
  \put(60,10){\circle{2}}  \put(64,9.3){\ldots}
  \put(80,10){\circle{2}}  \put(81,10){\line(1,0){18}}
  \put(97,8){$\scriptstyle{\times}$}  \put(101,10){\line(1,0){18}}
  \put(120,10){\circle{2}}  \put(120,9){\line(1,0){16}}
  \put(140,10){\circle{2}}  \put(120,11){\line(1,0){16}}
   \put(132,7){$>$}
\end{picture}
\quad\quad
\begin{picture}(200,40)
\put(50,25){Non-Levi}
  \put(4,20){\circle{2}}  \put(4,0){\circle{2}}
  \put(19.2,10.7){\line(-5,3){14}}  \put(19.2,9.3){\line(-5,-3){14}}
  \put(20,10){\circle{2}}  \put(21,10){\line(1,0){18}}
  \put(37,8){$\scriptstyle{\times}$}  \put(41,10){\line(1,0){18}}
  \put(60,10){\circle{2}}  \put(64,9.3){\ldots}
  \put(80,10){\circle{2}}  \put(81,10){\line(1,0){18}}
  \put(97,8){$\scriptstyle{\times}$}  \put(101,10){\line(1,0){18}}
  \put(120,10){\circle{2}}     \put(132,7){$>$}\put(120,9){\line(1,0){16}}\put(120,11){\line(1,0){16}}
  \put(140,10){\circle{2}}
\put(128,17){$\scriptstyle{ +1}$}
\put(0,9){$\scriptstyle{\mp 1}$}
\put(160,9){$\g=\mathfrak{so}(2n+1)$}
\end{picture}
$$
$$
\begin{picture}(150,40)
  \put(0,10){\circle{2}}  \put(1,10){\line(1,0){18}}
  \put(17,8){$\scriptstyle{\times}$}  \put(41,10){\line(1,0){18}}
  \put(40,10){\circle{2}}  \put(21,10){\line(1,0){18}}
  \put(60,10){\circle{2}}  \put(64,9.3){\ldots}
  \put(80,10){\circle{2}}  \put(81,10){\line(1,0){18}}
  \put(97,8){$\scriptstyle{\times}$}  \put(101,10){\line(1,0){18}}
  \put(120,10){\circle{2}}
  \put(136,20){\circle{2}}  \put(135,0){\circle{2}}
  \put(121.2,10.7){\line(5,3){14}}  \put(121.2,9.3){\line(5,-3){14}}
\put(133,8){$\scriptstyle{\pm 1}$}
\end{picture}
\quad\quad
\begin{picture}(200,40)
  \put(4,20){\circle{2}}  \put(4,0){\circle{2}}
  \put(19.2,10.7){\line(-5,3){14}}  \put(19.2,9.3){\line(-5,-3){14}}
  \put(20,10){\circle{2}}  \put(21,10){\line(1,0){18}}
  \put(37,8){$\scriptstyle{\times}$}  \put(41,10){\line(1,0){18}}
  \put(60,10){\circle{2}}  \put(64,9.3){\ldots}
  \put(80,10){\circle{2}}  \put(81,10){\line(1,0){18}}
  \put(97,8){$\scriptstyle{\times}$}  \put(101,10){\line(1,0){18}}
 \put(120,10){\circle{2}}
  \put(136,20){\circle{2}}  \put(135,0){\circle{2}}
  \put(121.2,10.7){\line(5,3){14}}  \put(121.2,9.3){\line(5,-3){14}}
\put(133,8){$\scriptstyle{\pm 1}$}
\put(0,9){$\scriptstyle{\mp 1}$}
\put(160,9){$\g=\mathfrak{so}(2n)$}
\end{picture}
$$
In other words, a non-Levi subgroup necessarily contains a semisimple orthogonal block of even dimension
rotating the eigenspace of $-1$ and, for even $N$, a semisimple orthogonal block rotating the eigenspace of $+1$.

 With a class $O$, we associate  an integer valued
vector $\nb=(n_i)_{i=1}^{\ell+2}$ subject to $\sum_{i=1}^{\ell+2}n_i=n$,
and a complex valued vector $\mub=(\mu_i)_{i=1}^{\ell+2}$. We assume that the coordinates of $\mub$ are all invertible,
with $\mu_i\not =\mu_j^{\pm 1}$ for $i<j\leqslant \ell$ and $\mu_i^2\not= 1$ for  $1\leqslant  i\leqslant \ell$. Finally,
we put  $\mu_{\ell+1}=-1$ and
 $\mu_{\ell+2}=1$. We reserve the special notation $m=n_{\ell+1}$ and $p=n_{\ell+2}$.

The initial point $o\subset O$ is fixed to the diagonal matrix with the entries
$$\underbrace{\mu_1,\ldots, \mu_1}_{n_1},\ldots, \underbrace{\mu_\ell,\ldots, \mu_\ell}_{n_\ell}
,\underbrace{-1,\ldots, -1}_{m},
\underbrace{1,\ldots, 1}_{P},\underbrace{-1,\ldots, -1}_m,
\underbrace{\mu_\ell^{-1},\ldots, \mu_\ell^{-1}}_{n_\ell},\ldots, \underbrace{\mu_1^{-1},\ldots, \mu_1^{-1}}_{n_1},
$$
where $P=2p$ if $N=2n$ and $P=2p+1$ if $N=2n+1$. We assume $m\geqslant 2$, also $p\geqslant 2$  for even $N$ and $p\geqslant 0$ for
odd $N$. The class with $p=1$, $m=2$ for even $N$ is a "boundary" case mentioned in the introduction, which is not considered here.

The stabilizer subgroup of the initial point $o\in O$ is the direct product
\be
\label{gr_K}
K=GL(n_1)\times\ldots\times GL(n_\ell)\times SO(2m)\times SO(P)
\ee
and it is determined solely by the vector $\nb$.
The integer $\ell$ counts the number of $GL$-blocks in $K$ of dimension $n_i$, $i=1,\ldots,\ell$,
while $m$ and $p$ are the ranks of the orthogonal blocks in $K$ corresponding to
the eigenvalues $-1$ and $+1$, respectively.
The specialization  $n_1=\ldots =n_\ell=0$ is formally encoded by $\ell=0$ and
 referred to as the {\em symmetric case}. Then (\ref{gr_K}) reduces to $SO(2m)\times SO(P)$, and
the class $O\simeq G/K$ to a symmetric space.

Let $\Mc_K$ denote the moduli space of conjugacy classes with
the fixed isotropy subgroup (\ref{gr_K}), regarded as Poisson spaces.
The set of all  $\ell+2$-tuples $\mub$  as above specified
parameterizes $\Mc_K$ although not uniquely. In particular, for even $N$ one can also choose the alternative
parametrization $\mu_{\ell+1}=1$, $\mu_{\ell+2}=-1$,
however it is compensated by the Poisson automorphism $A\mapsto -A$.
 Therefore, the subset $\hat \Mc_K$ of $\mub$
with fixed $\mu_{\ell+1}=-1$ and $\mu_{\ell+2}=1$ can be used for parametrization
of $\Mc_K$ (which is still not one-to-one).

The conjugacy class $O$ associated with $\mub$ and $\nb$ is specified by the set of equations
\be
(A-\mu_1)\ldots (A-\mu_\ell)
(A+1)(A-1)(A-\mu_\ell^{-1})\ldots (A-\mu_1^{-1})=0,
\label{min_pol_cl}
\\
\Tr(A^k)=\sum_{i=1}^\ell n_i(\mu_i^{k}+\mu_i^{-k})+2m(-1)^k+P, \quad k=1,\ldots, N,
\label{tr_cl}
\ee
where the matrix multiplication in the first line is understood. This system is polynomial
in the matrix entries $A_{ij}$ and defines an ideal of $\C[\End(\C^N)]$ vanishing on $O$.
\begin{thm}
The system of polynomial relations (\ref{min_pol_cl}) and (\ref{tr_cl}) along
with the defining relations of the group (\ref{ideal_group})  generates the defining ideal of the
class $O\subset SO(N)$.
\label{prop_clas_so}
\end{thm}
\begin{proof}
The proof is similar to the symplectic case worked out in \cite{M5}, Theorem 2.3. It boils down to checking
the rank of Jacobian of the system (\ref{min_pol_cl}), (\ref{tr_cl}), and (\ref{ideal_group}).
\end{proof}

\section{Quantum orthogonal groups}
\label{QOG}
The quantum group $U_\hbar(\g)$ is  a deformation of the universal enveloping
algebra $U(\g)$ along the parameter $\hbar$ in the class of Hopf algebras, \cite{D}.
By definition, it is a topologically free $\C[\![\hbar]\!]$-algebra. Here and further
on,  $\C[\![\hbar]\!]$ is the local ring of formal power series in $\hbar$.

Let $R$ and  $R^+$ denote respectively the root system and the set of positive roots of the
orthogonal Lie algebra $\g$.
Let  $\Pi^+=(\al_1,\al_1,\ldots, \al_{n})$ be the set of simple positive roots.
They can be  conveniently expressed through an orthonormal basis $(\ve_i)_{i=1}^n$ with
respect to the canonical symmetric inner form $(\>.\>,\>.\>)$ on the linear span of $\Pi^+$:
$$
\begin{array}{llllllrrr}
\al_i&=&\ve_i-\ve_{i+1}, & i=1,\ldots, n-1,&\al_n =\ve_{n-1}+\ve_{n}, & \g=\s\o(2n),\nn\\
\al_i&=&\ve_i-\ve_{i+1}, & i=1,\ldots, n-1,&\al_n =\ve_{n}, & \g=\s\o(2n+1).\nn
\end{array}
$$
Given a reductive subalgebra $\f\subset \g$ such that $\f \supset \h$, we label its root subsystem with subscript $\f$,
as well as the set of positive and simple positive roots: $R_\f$, $R_\f^+$, $\Pi^+_\f$.
We reserve the notation $\g_k$ for the orthogonal subalgebra of
rank $k\leqslant n$ corresponding to the positive roots $\Pi^+_{\g_k}=\{\al_{n-k+1}, \ldots, \al_n\}$.

Denote by $\h$ the  dual vector space to the linear span $\C\Pi^+$.
The inner product establishes a linear isomorphism between the $\C\Pi^+$ and $\h$.
We define $h_\la\in \h$
for every $\la\in \h^*=\C\Pi^+$ to be its image under this isomorphism:  $\mu(h_\la)=(\la,\mu)$ for all $h\in \h$.

The vector space $\h$ generates a commutative subalgebra $U_\hbar(\h)\subset U_\hbar(\g)$ called
the Cartan subalgebra.
The quantum group $U_\hbar(\g)$ is a $\C[\![\hbar]\!]$-algebra generated by simple root vectors
$
e_{\mu}, f_{\mu}
$ (Chevalley generators),
and $h_{\mu}\in \h$ (Cartan generators),
$\mu\in \Pi^+$. Both $U_\hbar(\h)$ and $U_\hbar(\g)$ are completed in $\hbar$-adic topology.
The Cartan  and Chevalley generators obey the commutation rule
$$
[h_{\mu},e_{\nu}]= (\mu,\nu) e_{\nu},
\quad
[h_{\mu},f_{\nu}]= -(\mu,\nu) f_{\nu},
$$
$$
[e_{\mu},f_{\nu}]=\delta_{\mu,\nu} \frac{q^{h_{\mu}}-q^{-h_{\mu}}}{q-q^{-1}},\quad \mu \in \Pi^+, \quad q=e^{\hbar}.
$$
Note with care that the denominator is independent of $\mu$, contrary to the usual definition
with   $q_\mu=e^{\hbar\frac{(\mu,\mu)}{2}}$, $\mu \in \Pi^+$, see e.g. \cite{CP}.
The difference comes from a rescaling of the Chevalley generators, which also
respects the Serre relations below. With our normalization,
the natural representation of $U_\hbar\bigl(\s\o(N)\bigr)$ on $\C^N$ is  determined by the classical matrix assignment on
the  generators, which is independent of $q$.

Let $a_{ij}=\frac{2(\al_i,\al_j)}{(\al_i,\al_i)}$,
$i,j=1,\ldots, n$, be the Cartan matrix and put $q_i:=q_{\al_i}$.
Define $[z]_q=\frac{q^z-q^{-z}}{q-q^{-1}}$ for any complex $z$, and
the $q$-binomial coefficients
$$
\left[
\begin{array}{cc}
n  \\ k
\end{array}
\right]_{q}
=
\frac{[n]_q!}{[k]_q![n-k]_q!},
\quad [0]_q!=1,
\quad
[n]_q!=[1]_q\cdot [2]_q\ldots [n]_q
$$
for $k,n\in\N$, $k \leqslant n$.
The positive Chevalley generators satisfy the quantum Serre relations
$$
\sum_{k=0}^{1-a_{ij}}(-1)^k
\left[
\begin{array}{cc}
1-a_{ij} \\
 k
\end{array}
\right]_{q_i}
e_{\al_i}^{1-a_{ij}-k}
e_{ \al_j}e_{ \al_i}^{k}
=0
.
$$
Similar relations hold for the negative Chevalley generators $f_\mu$.

The comultiplication $\Delta$ and antipode $\gm$ are defined on the generators by
$$
\Delta(h_\mu)=h_\mu\tp 1+1\tp h_\mu, \quad \gm(h_\mu)=-h_\mu,
$$
$$
\Delta(e_\mu)=e_\mu\tp 1+q^{h_\mu}\tp e_\mu, \quad \gm(e_\mu)=-q^{-h_\mu}e_\mu,
$$
$$
\Delta(f_\mu)=f_\mu\tp q^{-h_\mu}+1\tp f_\mu, \quad \gm(f_\mu)=-f_\mu q^{h_\mu},
$$
for all $\mu\in \Pi^+$.
The counit homomorphism $\ve\colon U_\hbar(\g)\to \C[\![\hbar]\!]$ annihilates $e_\mu$, $f_\mu$, $h_\mu$.
As in \cite{M5}, our comultiplication is opposite to the comultiplication used in \cite{CP}.

Besides the Cartan subalgebra $U_\hbar(\h)$, the quantum group $U_\hbar(\g)$ contains the following Hopf subalgebras.
 The positive and negative
Borel subalgebras $U_\hbar(\b^\pm)$ are generated over $U_\hbar(\h)$ by, respectively,
$\{e_\mu\}_{\mu\in \Pi^+}$ and $\{f_\mu\}_{\mu\in \Pi^+}$ as left (right) regular $U_\hbar(\h)$-modules. For any
 root subsystem in $R$ the associated Levi subalgebra $U(\l)$ is quantized to a Hopf subalgebra
$U_\hbar(\l)$, along with the parabolic subalgebras $U_\hbar(\p^\pm)$ generated by $U_\hbar(\b^\pm)$
over $U_\hbar(\l)$.

Let $U_q(\h)$ denote the subalgebra in $U_\hbar(\g)$ generated by the exponentials $t^{\pm}_{\al_i}=q^{\pm h_{\al_i}}$, $\al_i\in \Pi^+$. By $U_q(\g)\subset U_\hbar(\g)$ we mean the Hopf subalgebra generated over $U_q(\h)$  by
$\{e_\mu, f_\mu\}_{\mu\in \Pi^+}$. The other mentioned subalgebras in $U_\hbar(\g)$ have their counterparts in $U_q(\g)$ and will be denoted with
the subscript $q$. Note that all these algebras
are considered over $\C[\![\hbar]\!]$ but are not completed in the $\hbar$-adic topology. Also the $\hbar$- and $q$-versions have  different Cartan subalgebras.

Quantum counterparts  $e_\mu, f_\mu \in U_\hbar(\g)$, $\mu\in R^+$, of higher root vectors are defined
through a reduced decomposition of the maximal element of the Weyl group, \cite{CP}. Contrary to the classical case, they depend on such a decomposition. Higher root vectors generate a Poincare-Birkhoff-Witt (PBW) basis in $U_\hbar(\g)$ over $U_\hbar(\h)$, \cite{CP}. This basis
establishes a linear isomorphism of the adjoint $\h$-modules $U_\hbar (\g)$ and $\U(\g)\tp \C[\![\hbar]\!]$.
This isomorphism enables the use of the same notation for $\h$-submodules
in $U(\g)$ and $U_\hbar(\g)$. For instance, by $\g\subset U_\hbar(\g)$ we understand the sum of $\h$ and the
linear span of $\{f_\mu, e_\mu\}_{\mu\in \Pi^+}$.

The triangular decomposition $\g=\n_\l^-\op \l \op \n_\l^+$ gives rise to the triangular factorization
\be
\label{tr_fac}
U_\hbar(\g)=U_\hbar(\n_\l^-) U_\hbar(\l)U_\hbar(\n_\l^+),
\ee
where $U_\hbar(\n_\l^\pm)$ are subalgebras in $U_\hbar(\b^\pm)$ generated by the
positive or negative root vectors from $\n^\pm_\l$, respectively, \cite{Ke}.
This factorization makes $U_\hbar(\g)$ a free $U_\hbar(\n_\l^-)-U_\hbar(\n_\l^+)$-bimodule generated by $U_\hbar(\l)$.
For the special case $\l=\h$, we denote
$\g_\pm=\n^\pm_\h$. Contrary to the classical universal enveloping algebras,
$U_\hbar(\n^\pm_\l)$ are not Hopf subalgebras in $U_\hbar(\g)$.

%The Serre relation of general linear group
%$$e_{\al_1}^2e_{\al_2}-(q+q^{-1})e_{\al_1}e_{\al_2}e_{\al_1}+e_{\al_2}e_{\al_1}^2=0$$ implies

\section{Auxiliary parabolic Verma module $\hat M_\la$}
We adopt certain  conventions concerning  representations of quantum groups, which are similar to \cite{M5}.
We assume that they are free modules over $\C[\![\hbar]\!]$ and their rank will
be referred to as  dimension.
Finite dimensional $U_\hbar(\g)$-modules are deformations of their classical counterparts, and we will drop the reference
to the deformation parameter in order to simplify notation. For instance, the natural
$N$-dimensional representation of $U_\hbar(\g)$  will be denoted simply by $\C^N$.

We shall deal with weight i.e. $U_\hbar(\h)$-diagonalizable,  modules. If $V$ is an $\h$-invariant subspace, we mean by $[V]_\al$
the subspace of weight $\al\in \h^*$.
We stick to  the additive parametrization of weights of $U_q(\g)$ facilitated by the embedding $U_q(\h)\subset U_\hbar(\h)$.
Under this convention, such weights belong to  $\frac{1}{\hbar} \h^*[\![\hbar]\!]$. Indeed, for any $\la\in \frac{1}{\hbar} \h^*[\![\hbar]\!]$
its values on the generators $t^{\pm}_{\al_i}$ are $q^{\pm\la(h_{\al_i})}\in \C[\![\hbar]\!]$, so $\la$ is well defined
on $U_q(\h)$. By reasons explained below, it is
sufficient for our needs to confine weights to the subspace $ \hbar^{-1}\h^*\op \h^* \subset \hbar^{-1}\h^*[\![\hbar]\!]$.

Let $L\subset G$ denote the Levi subgroup
$$
L=GL(n_1)\times \ldots \times GL(n_\ell)\times GL(m)\times SO(P).
$$
It is a maximal Levi subgroup of $G$ among those contained in $K$, cf.  (\ref{gr_K}) (the other
one is obtained by reducing $SO(P)$ to $GL(p)$).
By $\l$ we denote
the Lie algebra of $L$. It is a reductive subalgebra in $\g$ of maximal rank $n$.

We denote by $\c_\l\subset \h$ the center of $\l$ and realize its dual $\c_\l^*$ as a subspace in $\h^*$ thanks to
the canonical inner product.
A element $\la\in \mathfrak{C}^*_{\l}= \hbar^{-1} \c_\l^*\op  \c_\l^*$ defines
a one-dimensional representation of $U_q(\l)$ denoted by  $\C_\la$. Its restriction to the Cartan subalgebra acts by the
assignment $q^{h_\al}\mapsto q^{(\al,\la)}$. Since $q=e^\hbar$, the pole in $\la$ is compensated, and the representation is correctly defined.
It extends to $U_q(\p^+)$ by nil on $\n_\l^+\subset \p_\l^+$. Denote by $\hat M_\la=U_q(\g)\tp_{U_q(\p^+)}\C_\la$ the parabolic Verma $U_q(\g)$-module
induced from $\C_\la$, \cite{Ja}. It plays an intermediate role in our construction: we are interested
in a quotient module $M_\la$, which can be defined for certain values of $\la$.

Regarded as a $U_q(\h)$-module, $\hat M_\la$ is isomorphic to $U_q(\n^-_\l)\tp \C_\la$, as follows from (\ref{tr_fac}). This implies
that $\hat M_\la$ are isomorphic as $\C[\![\hbar[\!]$-modules  for all $\la$.
Let $v_\la$ denote the image of $1\tp 1$ in  $\hat M_\la$. It generates $\hat M_\la$ over $U_q(\g)$
and carries the highest weight $\la$.
For any sequence of Chevalley generators $f_{\al_{k_1}},\ldots, f_{\al_{k_m}}$ we call the product
$f_{\al_{k_1}}\ldots f_{\al_{k_m}}v_\la\in \hat M_\la$  Chevalley monomial or simply monomial.

Along with $\hat M_\la$, we consider the {\em right} $U_q(\g)$-module $\hat M_\la^\star=\C_\la\tp_{U_q(\p^-)}U_q(\g)$.
Here $\C_\la$ supports the $1$-dimensional representation of $U_q(\p^-)$ which
extends the $U_q(\l)$-representation  by nil on $\n_\l^-\subset \p_\l^-$.
As a $U_q(\h)$-module, it is isomorphic to $\C_\la\tp U_q(\n^+_\l)$
generated $v_\la^\star =1\tp 1$.
 Given a monomial
$v=f_{\al_{k_1}}\ldots f_{\al_{k_m}}v_\la$ we define  $v^\star$ to be the monomial
$v_\la^\star e_{\al_{k_m}}\ldots e_{\al_{k_1}}\in \hat M_\la^\star$.
There is a bilinear  pairing (Shapovalov form) between $\hat M_\la^\star$ and  $\hat M_\la$. It is determined by the following
requirements: {\em i})
 $\langle xu,y\rangle =\langle x,uy\rangle$ for all $x\in \hat M_\la^\star$, $y\in \hat M_\la$, $u\in U_q(\g)$,
{\em ii}) $v_\la^\star$ is orthogonal to all vectors of
weight lower than $\la$,
{\em iii}) it is normalized to $\langle v_\la^\star,v_\la \rangle =1$.

As in \cite{M5}, we introduce a subspace of weights   that we use for the parametrization of
$\Mc_K$, the moduli space  of conjugacy classes with fixed $K$.
Define $\mathcal{E}_i\in \h^*$, $i=1,\ldots,\ell+2$, by
$$\mathcal{E}_1=\ve_1+\ldots +\ve_{n_1}, \quad \mathcal{E}_2=\ve_{n_1+1}+\ldots +\ve_{n_1+n_2},\quad \ldots,\quad\mathcal{E}_{\ell+2}=
\ve_{n-p+1}+\ldots+\ve_n.$$
The vector space  $\c^*_\l$ is formed by $\la=\sum_{i=1}^ {\ell+2} \La_i \E_i$ with $\La_i\in \C$ and $\La_{\ell+2}=0$.
Put $\mu^0_k=e^{2\La_k}$, for $k=1,\ldots, \ell+2$. Let $\c_{\l,reg}^*$ denote the set of all weights $\la\in \c_\l^*$ such that $\mu^0_k \not= (\mu^0_j)^{\pm1}$ for $k\not =j$.
Denote by $\c_{\k}^*\subset \c_\l^*$ its subset determined by $\mu^0_{\ell+1}=-1$
and by $\c_{\k,reg}^*\subset \c_{\k}^*$  the subspace of such $\la\in \c_\l^*$ that $\mu^0_k\not =(\mu^0_j)^{\pm1}$
for $k,j=1,\ldots, \ell+2$, $k\not =j$. Obviously $\c_{\k,reg}^*$ is dense in $\c_{\k}^*$.
Finally, we introduce $\mathfrak{C}^*_{\k,reg}\subset\mathfrak{C}^*_{\k}\subset\frac{1}{\hbar}\c_\l^*\oplus \c_\l^*$ by
setting
$\mathfrak{C}^*_{\k}=\hbar^{-1} \c_{\k}^*-\frac{P}{2}\E_{\ell+1}$ and $\mathfrak{C}^*_{\k,reg}=\hbar^{-1} \c_{\k,reg}^*-\frac{P}{2}\E_{\ell+1}$.
Clearly $\mathfrak{C}^*_{\k,reg}$ is dense in $\mathfrak{C}^*_{\k}$.
By construction, all weights from $\mathfrak{C}^*_{\k}$ satisfy
$q^{2(\al_{n-p},\la)}=-q^{- P}$.

Note that the
 vector $\mub^0=(\mu^0_i)$ for $\la \in \c_{\k,reg}^*$ belongs
to $\hat \Mc_K$ covering $\Mc_K$, and all points in $\hat \Mc_K$ can be obtained this way.
\subsection{Some auxiliary technicalities}
In this  section we introduce some constructions which we use further on. They involve
the quantum subgroup $U_q\bigl(\g\l(n)\big)$ in $U_\hbar(\g)$ corresponding to the roots $\{\al_i\}_{i=1}^{n-1}$.
Its  negative Chevalley generators obey the Serre relations
\be
\label{Serre}
f_{\al_i}^2f_{\al_{i\pm 1}}-(q+q^{-1})f_{\al_i}f_{\al_{i\pm 1}}f_{\al_i}+f_{\al_{i\pm 1}}f_{\al_i}^2=0, \quad
[f_{\al_i},f_{\al_j}]=0,
\ee
for all feasible  $i$, $j$, and $|i-j|>1$ (the positive generators satisfy similar relations).
These relations will be heavily used in what follows.

Let us fix the Levi subalgebra $\l=\g\l(2)\oplus \s\o(N-4)$, which corresponds to $\ell=0$, $m=2$, $p=n-2>0$,
and let $\hat M_\la$ be a parabolic Verma module relative to $U_q(\l)$. Note that $f_{\al}$ kills the  generator $v_\la\in \hat M_\la$ unless $\al =\al_2$.
Put $\kappa=p=n-2$ if $N=2n$ and $\kappa=p+1=n-1$ if $N=2n+1$.
Introduce the element
$\omega=f_{\al_\kappa} \ldots f_{\al_2} v_\la\in \hat M_\la$ and also
$\omega=v_\la$ for $p=n-2=0$.  This vector participates
in a basis, which we use for calculation of a singular vector in $\hat M_\la$ in the subsequent sections.
It is
constructed solely out of the $\g\l(n)$-generators and
 features the following.
\begin{lemma}
\label{omega}
Suppose that $3\leqslant \kappa$. Then $\omega$ is annihilated by
$f_{\al_i}$, $3\leqslant i \leqslant \kappa$.
\end{lemma}
\begin{proof}
Let $\sim$ denote equality up to a scalar factor.
Assuming $3\leqslant i < \kappa$, we get
$$
f_{\al_i}\omega=f_{\al_\kappa}\ldots f_{\al_i}f_{\al_{i+1}}f_{\al_i} \ldots f_{\al_2} v_\la
\sim f_{\al_\kappa}\ldots (f_{\al_i}^2f_{\al_{i+1}}+f_{\al_{i+1}}f_{\al_i}^2) \ldots f_{\al_2} v_\la,
$$
by the Serre relations (\ref{Serre}). The rightmost dots contain  generators with numbers strictly less than $i$.
Since they commute with $f_{\al_{i+1}}$, it can be pushed to the right in the first summand, where it kills $v_\la$.
The second summand is equal to
$
f_{\al_\kappa}\ldots f_{\al_i}^2f_{\al_{i-1}} \ldots  v_\la
$.
Again, using the Serre relation for $f_{\al_i}^2f_{\al_{i-1}}$ we can place  at least one factor $f_{\al_i}$
on the right of $f_{\al_{i-1}}$ and push it freely  further to the right. This kills the second summand.

If $i=\kappa$, then we have $f_{\al_\kappa}\omega=f_{\al_\kappa}^2 f_{\al_{\kappa-1}}\ldots f_{\al_2} v_\la$.
Using  (\ref{Serre}), at least one copy of  $f_{\al_\kappa}$ can be pushed through $f_{\al_{\kappa-1}}$ to the right
and further on till it kills $v_\la$.
\end{proof}
\begin{lemma}
\label{omega_singular}
The vector $\omega$ is annihilated by
$e_{\al_i}$, $\al_i\in \Pi^+$, $i\not = \kappa$.
\end{lemma}
\begin{proof}
Obviously, $\omega$ is annihilated by $e_{\al_i}$ if $i=1$ and  $i> \kappa$.
Applying $e_{\al_i}$ with  $2\leqslant i \leqslant \kappa-1$ to $\omega$ yields
$
f_{\al_\kappa} \ldots f_{\al_{i+1}}\ldots v_\la
$
up to a scalar multiplier.
Here the dots on the right stand for the generators with numbers strictly less than $i$. Since they commute
with $f_{\al_{i+1}}$, the latter can be pushed to the right, where it kills $v_\la$.
\end{proof}
Remark that $\omega$ is a non-zero vector of  weight
$\la-\ve_2+\ve_{\kappa+1}$. Indeed, one can check that
$\dim [\hat M_\la]_{\la-\ve_2+\ve_{\kappa+1}}=1$ and all other monomials of this weight turn zero.

Further we present another auxiliary construction, which also involves only the $\g\l(n)$-generators.
Suppose that $3\leqslant n$ and introduce vectors  $y_k\in \hat M_\la$, $k=2,\ldots,n-1,$
 by
$$
y_2=[f_{\al_1},f_{\al_2}]_a f_{\al_{{2}}}v_\la,
\quad y_k=f_{\al_{k}}\ldots f_{\al_{{3}}}[f_{\al_1},f_{\al_2}]_a f_{\al_{k}}\ldots f_{\al_{{2}}}v_\la,
\quad 3\leqslant k\leqslant n-1,
$$
where $[X,Y]_a=XY-aYX$. Here and further on we set the parameter $a$ equal to $q+q^{-1}$.
\begin{lemma}
For all $k=2,\ldots,n-1$, one has $y_k=0$.
\label{y_zeros}
\end{lemma}
\begin{proof}
For $k=2$ we find
$[f_{\al_1},f_{\al_2}]_a f_{\al_{{2}}}v_\la=f_{\al_1}f_{\al_2}f_{\al_{{2}}}v_\la-af_{\al_2}f_{\al_1}f_{\al_{{2}}}v_\la=0$
by the Serre relation (\ref{Serre}).
For higher $k$ we use induction. Suppose the lemma is proved for some $k\geqslant 2$.
Then
\be
y_{k+1}&=&f_{\al_{k+1}}f_{\al_{k}}\ldots f_{\al_{{3}}}[f_{\al_1},f_{\al_2}]_a f_{\al_{k+1}}f_{\al_{k}}\ldots f_{\al_{{2}}}v_\la
\nn\\
&=&(f_{\al_{k+1}}f_{\al_{k}}f_{\al_{k+1}})\ldots f_{\al_{{3}}}[f_{\al_1},f_{\al_2}]_a f_{\al_{k}}\ldots f_{\al_{{2}}}v_\la.
\nn
\ee
The term in the brackets produces
$a^{-1}(f_{\al_{k+1}}^2 f_{\al_{k}}+f_{\al_{k}}f_{\al_{k+1}}^2)$
through (\ref{Serre}).
The first term is zero by the induction assumption. The second term is zero too, because
$f_{\al_{k+1}}^2$ can be pushed to the right till it meets the second copy of $f_{\al_{k}}$. By the Serre relation,
one factor $f_{\al_{k+1}}$ can be pushed through $f_{\al_{k}}$ to the right. Then it proceeds  freely till
it kills $v_\la$. This proves the statement.
\end{proof}
Remark that the case $N=5$, $m=2$, $p=0$ is excluded from this construction, and  $y_2\not =0$ then.

\subsection{The module $\hat M_\la$ for  $\l=\g\l(2)\oplus \s\o(P)$}
A substantial part of this theory is captured by  the
special case of symmetric conjugacy classes. That accounts for the fact that the difference between $K$ and $L$
is confined within the orthogonal blocks of $K$. Because of that, we start from the symmetric
case, when the stabilizer $\k$ consists of two simple orthogonal blocks of ranks $m$ and $p$.
Furthermore, the general symmetric case can be readily derived from
the specialization $m=2$ (note that $\s\o(4)$ is the smallest semisimple orthogonal algebra of even dimension).
For this reason, we start with $\k=\s\o(4)\oplus \s\o(N-4)$, $\l=\g\l(2)\oplus \s\o(N-4)$.

Observe that $U(\k)$ is generated over $U(\l)$ by a pair of root vectors $e_\dt, f_\dt$,
where
$$
\delta=\al_1+2\sum_{i=2}^{n-2}\al_i+\al_{n-1}+\al_n, \quad \g=\s\o(2n),\quad
\delta=\al_1+2\sum_{i=2}^{n}\al_i, \quad \g=\s\o(2n+1).
$$
Namely,
$\k=\m^-\oplus \l\oplus \m^+$, where $\m^-=\ad(\l)(f_\dt)$ and $\m^+=\ad(\l)(e_\dt)$
are abelian Lie subalgebras. The algebra $U(\k)$ features the triangular decomposition
$U(\k)=U(\m^-)\times U(\l)\times  U(\m^+)$.

In the symmetric case under consideration, the weight $\la$ satisfies the conditions
$(\al_i,\la)=0$ for all $i$ but $i=2$. Therefore, $\hat M_\la$ is parameterized by scalar $(\al_2,\la)$.
Its highest weight vector $v_\la$ is annihilated  by all $e_{\al_i}$ and all  $f_{\al_i}$ except for  $f_{\al_2}$.
Regarding $\hat M_\la$ as a $U_\hbar(\g_-)$-module consider its classical limit $\hat M_\la/\hbar\hat M_\la$.
It  is generated by the  root vectors
$$
f_{\ve_1\pm \ve_i}, \> f_{\ve_2\pm \ve_i}, \>  f_{\ve_1+ \ve_2}, \>  f_{\ve_1}, \>  f_{\ve_2} \in \n^-_\l,
$$
where $i=3,\ldots, n$, and $f_{\ve_1}$, $f_{\ve_2}$ are present only when $N$ is odd.
Therefore, modulo $\hbar$, the weight space $[\hat M_\la]_{\la-\dt}$, has the basis of $N-3$ elements
$$
f_{\ve_1\pm \ve_i}f_{\ve_2\mp \ve_i}v_\la,\quad f_{\ve_1+\ve_2}v_\la, \quad f_{\ve_1}f_{\ve_2}v_\la,
$$
where last term counts for odd $N$. Since $\hat M_\la$ is $\C[\![\hbar]\!]$-free, $\dim [\hat M_\la]_{\la-\dt}=N-3$.

We intend to  calculate a singular vector $v_{\la-\dt}\in  [\hat M_\la]_{\la-\dt}$ where $\la$ allows for it. Singular means that $v_{\la-\dt}$ lies
in the kernel of all $e_{\al}\in \g_+$. In order to facilitate the calculations, we need to choose a suitable basis.
Notice that in the classical limit the subspace of weight $\la-\dt+\al_1=\la-2\ve_2$ (the image $e_{\al_1}[\hat M_\la]_{\la-\dt}$ for
generic $\la$) has a basis
$$
f_{\ve_2\pm \ve_i}f_{\ve_2\mp \ve_i}v_\la, \quad f_{\ve_2}f_{\ve_2}v_\la,
$$
where the last term counts for odd $N$. Therefore,  $\ker e_{\al_1}|_{[M_\la]_{\la-\dt}}=[\ker e_{\al_1}]_{\la-\dt}$ has dimension $n-1$.

We consider the PBW basis being not particularly convenient for our purposes and introduce another basis in
$[\ker e_{\al_1}]_{\la-\dt}$.
First we do it for the lowest dimensions $N=5,7,8$. For $N=5$ we have only one vector
 $$
 x_2=[f_{\al_1},f_{\al_2}]_a f_{\al_2} v_\la,
 $$
for $N=7$ there are two vectors
 $$
 x_2=[f_{\al_1},f_{\al_2}]_a f_{\al_3}(f_{\al_3}\omega),\quad
  x_3=f_{\al_3}[f_{\al_1},f_{\al_2}]_a  (f_{\al_3}\omega),\quad \omega = f_{\al_2}v_\la.
$$
There are three vectors for $N=8$:
$$
x_{2}=[f_{\al_1},f_{\al_2}]_a  (f_{\al_{3}}f_{\al_{4}} \omega),
\quad
x_{3}=f_{\al_{3}}[f_{\al_1},f_{\al_2}]_a f_{\al_4} \omega,
\quad
x_{4}=f_{\al_{4}} [f_{\al_1},f_{\al_2}]_a f_{\al_{3}} \omega,
\quad
\omega=f_{\al_2} v_\la.
$$
Assuming $N>8$, define $n-1$ vectors $x_i\in [\hat M_\la]_{\la-\dt}$ by
\be
x_{2}&=&[f_{\al_1},f_{\al_2}]_a f_{\al_{3}}
\stackrel{<}{\ldots} f_{\al_{n}}(f_{\al_{n}}\omega),
\nn\\
x_{i}&=&f_{\al_{i}}\stackrel{>}{\ldots} f_{\al_{3}}[f_{\al_1},f_{\al_2}]_a f_{\al_{i+1}}
\stackrel{<}{\ldots} f_{\al_{n}}(f_{\al_{n}}\omega), \quad i=3,\ldots,n,
\nn
\ee
for $N=2n+1$, and by
\be
x_{2}&=&[f_{\al_1},f_{\al_2}]_a f_{\al_3}
\stackrel{<}{\ldots} f_{\al_{n-2}} (f_{\al_{n-1}}f_{\al_{n}} \omega),
\nn\\
x_{i}&=&f_{\al_{i}}\stackrel{>}{\ldots} f_{\al_{3}}[f_{\al_1},f_{\al_2}]_a f_{\al_{i+1}}
\stackrel{<}{\ldots} f_{\al_{n-2}} (f_{\al_{n-1}}f_{\al_{n}} \omega),\quad i=3,\ldots,n-2,
\nn\\
x_{n-1}&=&f_{\al_{n-1}} f_{\al_{n-2}} \stackrel{>}{\ldots}f_{\al_3} [f_{\al_1},f_{\al_2}]_a f_{\al_n} \omega,
\nn\\
x_{n}&=&f_{\al_{n}} f_{\al_{n-2}} \stackrel{>}{\ldots}f_{\al_3} [f_{\al_1},f_{\al_2}]_a f_{\al_{n-1}} \omega,
\nn
\ee
for $N=2n$. The products are ordered with respect to the root numbers as indicated. The element $\omega$ for all $N$ is defined in the previous section. We emphasize that the generators in the parenthesis stay within as $i$ varies,
while other generators are permuted as specified.

The following lemma accounts for the choice of the commutator parameter $a$.
\begin{lemma}
The  vectors $x_i$, $i=2,\ldots,n$, belong to  $\ker e_{\al_1}\subset \hat M_\la$.
\label{al1}
\end{lemma}
\begin{proof}
Applying $e_{\al_1}$ to $x_i$ we get
$$
e_{\al_1}x_i \sim \ldots [q^{h_{\al_1}}-q^{-h_{\al_1}},f_{\al_2}]_a\ldots \omega=
\bigl((q^2-q^{-2})-a(q-q^{-1})\bigr)\ldots f_{\al_2}\ldots \omega=0.
$$
Indeed, observe that $h_{\al_1}$ commutes with everything between the commutator and $\omega$.
Further, the weight of $\omega$ is $\la-\ve_{2}+\ve_{n-1}$ for $N=2n\geqslant 8$, $\la-\ve_{2}+\ve_{n}$ for $N=2n+1\geqslant 7$, and $\la-\ve_{2}$ for $N=5$.
This produces the vanishing scalar factor in the brackets.
\end{proof}
As we already mentioned, the total dimension of $[\hat M_\la]_{\la-\dt}$ is equal to
$N-3$.  Every vector $x_i$ contains the commutator $[f_{\al_1},f_{\al_2}]_a$ thus involving
two Chevalley monomials. Overall $\{x_i\}_{i=2}^n$ involve $2n-2$  monomials of weight $\la-\delta$.
This is equal to $\dim [\hat M_\la]_{\la-\dt}$ for odd $N$, but greater by $1$ for even $N$.
However,
$$
f_{\al_1}f_{\al_{n-1}}f_{\al_p}\ldots f_{\al_2}f_{\al_{n}}\omega
\sim f_{\al_1}f_{\al_{n}}f_{\al_{n-1}}f_{\al_p}^2\ldots f_{\al_2}^2\omega
\sim
f_{\al_1}f_{\al_{n}}f_{\al_p}\ldots f_{\al_2}f_{\al_{n-1}}\omega.
$$
Therefore, there are effectively  $2n-3$ Chevalley monomials participating in $\{x_i\}_{i=2}^n$ for $N=2n$, as required.

Our search for $v_{\la-\dt}$ will be restricted to the subspace $[\ker e_{\al_1}]_{\la-\dt}$, so  $n-1$ vectors  $x_i$ annihilated by $e_{\al_1}$ are just
enough to form a basis. Next we prove a lemma, which  is crucial for checking the linear independence of  $x_i$.
Introduce  vectors $x'_i\in [\hat M_{\la}]_{\la-\dt+\al_i}$ for  $i=2,\ldots,n$ as follows:
$x'_{2}$ is obtained from $x_{2}$ by replacing the commutator $[f_{\al_1},f_{\al_2}]_a$ with $f_{\al_1}$;
to get $x'_{i}$ for $i>2$, we remove the leftmost copy of $f_{\al_i}$ from $x_{\al_i}$.
One can see that $e_{\al_i}x_i \sim x'_i$  for $i=2,\ldots,n$.
\begin{lemma}
For all $i=2,\ldots,n$, $x_{i}'\not =0$.
\label{e_x_diag}
\end{lemma}
\begin{proof}
Observe that $\dim[\hat M_\la]_{\la-\dt+\al_2}=1$ and
$\dim[\hat M_\la]_{\la-\dt+\al_i}=2$, where $i=3,\ldots ,n$ (for this verification, one can use the classical PBW basis).
Also, notice that $\dim[\hat M_\la]_{\la-\dt+\al_i+\al_1}=1$ for such $i$.
Consider the Chevalley monomials $x''_i$ of weights $\la-\dt+\al_i+\al_1$, $i=3,\ldots,n$, obtained from $x'_{i}$
by replacing the commutator $[f_{\al_1},f_{\al_2}]_a$ with $f_{\al_2}$.
 Using Lemma \ref{omega_singular}, one can easily calculate the matrix elements  of the Shapovalov pairing
$$
\langle {x'}^\star_{2}, x'_{2}\rangle =\langle \omega^\star, \omega\rangle,
\quad
\langle {x''}^\star_{i}, x''_{i}\rangle =\frac{q^{(\al_2,\la)-1}-q^{-(\al_2,\la)+1}}{q-q^{-1}}\langle \omega^\star, \omega\rangle ,
\quad i>2,
$$
and $\langle \omega^\star, \omega\rangle =\frac{q^{(\al_2,\la)}-q^{-(\al_2,\la)}}{q-q^{-1}}$.
This calculation proves that $x_{2}'$ and  $x''_i$  do not vanish for generic $\la$ and hence for all $\la$
(the $U_q(\g_-)$-module $\hat M_\la $ is isomorphic to $U_q(\g_-)/\sum_{\al\in \Pi^+_\l} U_q(\g_-)f_\al$ and hence "independent of $\la$").

Further, there are exactly two ways to construct a monomial of weight ${\la-\dt+\al_i}$, $i=3,\ldots,n$,
 out of $x''_i$: either placing $f_{\al_1}$ before or after the leftmost $f_{\al_2}$ (note that $f_{\al_2}$ is the only generator which does not commute with $f_{\al_1}$).
This gives two independent monomials participating in  $x'_{i}$, $i=3,\ldots ,n$. Consequently,
$x'_{i}$ do not vanish.
\end{proof}
Note that the vectors $x_i$ can be labeled with the simple roots of the subalgebra $\g_{n-1}=\g_{p+1}\subset \g$ via
the assignment $\al_i\mapsto x_i$, $i=2,\ldots,n$.
The next proposition provides qualitative information about the action of positive Chevalley generators on the
system $\{x_i\}\subset  \ker e_{\al_1}$.
\begin{propn}
For all $\al,\al_i\in \Pi^+_{\g_{p+1}}$ such that $(\al,\al_i)=0$ the generator $e_\al$ annihilates $x_{i}$.
If $(\al_j,\al_i)\not =0$, then $e_{\al_j} x_{i}\sim x'_{j}$.
\label{e_x}
\end{propn}
\begin{proof}
Suppose first that $N$ {\em  is even} and put $\al_{n-1}=\mu$, $\al_{n}=\nu$. Denote also $x_{\mu}=x_{n-1}$ and $x_{\nu}=x_{n}$.
Up to a scalar multiplier, $e_\nu x_\mu$ is equal to
$y_{n-1}$, which is zero due to Lemma \ref{y_zeros}. Further, observe that  $e_\mu x_{i}$, for  $i<p$,  contains the factor  $f_{\al_p}f_\nu f_{\al_p}$ producing
$f_{\al_p}^2f_\nu$ and $f_\nu f_{\al_p}^2$ via the Serre relation (\ref{Serre}). In the first term, the generator $f_\nu$ goes freely to the right and kills $v_\la$.
The second term gives rise to the factor $f_{\al_p}\omega$, which is nil by Lemma \ref{omega}.
Due to the symmetry  between the roots $\mu$ and $\nu$, this also proves
$e_\mu x_\nu=0$ and $e_\mu x_{i}=0$ for $i<p$.

By Lemma \ref{omega_singular}, $e_{\al_i}$ kills $\omega$, once $2\leqslant i< p$. Therefore, such $e_{\al_i}$ knock
 out the factor $f_{\al_i}$ from
$x_\mu=f_\mu f_{\al_p}\ldots f_{\al_{i+1}}f_{\al_{i}}\ldots [f_{\al_1},f_{\al_2}]\ldots \omega$  releasing
 $f_{\al_{i+1}}$ next to the left. The latter can be pushed to the right till it meets
$\omega$ and annihilates it by Lemma \ref{omega}. Hence $e_{\al_i}x_\mu=e_{\al_i}x_\nu=0$ for $2\leqslant i< p$.
Similar effect is produced by the action of $e_{\al_i}$ on $x_{j}$ for $3\leqslant i+1<j\leqslant p$.
If $3\leqslant j+1<i\leqslant p$, the
vector $e_{\al_i}x_{j}$ contains the factor
$
f_{\al_{i-1}}f_{\al_{i+1}}\ldots  \omega=
\ldots  f_{\al_{i-1}}\omega,
$
which is zero due to Lemma \ref{omega}. This completes the proof of the first assertion for even $N$.

Now suppose that $N$ is {\em odd}. There is nothing to prove if $p=0$, as there is only one vector, $x_2$.
So we assume $p>0$.
Let us check that $e_{\al_i}x_{j}=0$ when $3\leqslant j+1<i\leqslant n$. Then $x_{j}$ has the structure
$
 \ldots [f_{\al_1},f_{\al_2}]f_{\al_{j+1}}\ldots f_{\al_{i-1}}f_{\al_{i}}\ldots (f_{\al_n}\omega).
$ Observe that $e_{\al_i}$
effectively acts only on the displayed copy of $f_{\al_{i}}$. The other copy is hidden in $\omega$ and can be neglected, because
 $e_{\al_i}$ kills $\omega$ if $i\not= n-1,n$, by Lemma \ref{omega_singular}. If $i=n-1$, then $f_{\al_n}e_{\al_i}\omega=0$ by similar arguments.  If $i=n$, then $e_{\al_i}x_{j}$ still comprises the factor $f_{\al_{n-1}} f_{\al_n}\omega$.
In all cases, $e_{\al_{i}}$ knocks out the leftmost $f_{\al_{i}}$ and releases the factor
$f_{\al_{i-1}}$ on the left, which goes freely to the right. It kills $\omega$ by Lemma \ref{omega} if $i<n$.
Still it  kills $f_{\al_n}\omega$ if $i=n$, and  the proof is similar to Lemma \ref{omega}.

If $3\leqslant i+1<j\leqslant n$, then $x_{j}=f_{\al_j}\ldots f_{\al_{i+1}}f_{\al_{i}}\ldots [f_{\al_1},f_{\al_2}]_a\ldots f_{\al_n}\omega$,
and $f_{\al_{i+1}}$ commutes with everything between $f_{\al_{i}}$ and $f_{\al_n}$. Further reasoning is similar to the case $j+1<i$,
with $f_{\al_{i- 1}}$ replaced by $f_{\al_{i+1}}$. Therefore,  $e_{\al_i}x_{j}=0$ for $i+1<j$. This completes the first part of the proposition for odd $N$.

The proof of the second statement becomes quite straightforward on examining the structure of $x_{i}$.
This is left for the reader as an exercise.
\end{proof}

\begin{corollary}
The vectors $\{x_{i}\}_{i=2}^n$  form a basis in $[\ker e_{\al_1}]_{\la-\dt}$.
\label{basis}
\end{corollary}
\begin{proof}
By Proposition \ref{e_x},  the operator $E=\sum_{i=2}^n e_{\al_i}$ sends the linear span  $X=\Span\{x_{i}\}\subset [\ker e_{\al_1}]_{\la-\dt}$  to the linear
span $X'=\Span\{x'_{i}\}$. By Lemma \ref{e_x_diag}, all $x'_{i}\not =0$, have different weights and hence independent.
 We shall see in the next section (cf. Proposition \ref{singular}) that $ \ker E|_X=\{0\}$ for generic  $\la$.
Hence $\{x_{i}\}_{i=2}^n$ are independent for generic $\la$. Thanks
to the $U_q(\g_-)$-module isomorphism $\hat M_\la \simeq U_q(\g_-)/\sum_{\al\in \Pi^+_\l} U_q(\g_-)f_\al$, they
are independent for all $\la$. This proves the statement,
since $\dim [\ker e_{\al_1}]_{\la-\dt}= n-1$.
\end{proof}
Remark that independence of $\{x_{i}\}_{i=2}^n$ affects uniqueness of $v_{\la-\dt} \in \Span \{x_{i}\}_{i=2}^n$, for special $\la$, but not
its existence.

\section{The module $M_\la$}
In this section we  construct the highest weight $U_q(\g)$-module $M_\la$
that supports quantization of the class $G/K$. We define it
as a quotient of $\hat M_\la$ by a proper submodule generated by a singular vector of certain
weight.
First we do it
for symmetric $G/K$ and afterwards extend the solution for
general $K$.
\subsection{The symmetric case}
Consider the simplest symmetric case  $m=2$, $p=n-2$.
In other words, assume  $\k=\s\o(4)\oplus \s\o(P)$, and $\l=\g\l(2)\oplus \s\o(P)$.

Define a vector $v_{\la-\dt}\in [\hat M_\la]_{\la-\dt}$ by
\newcommand{\er}{\mathrm{e}}
$$
v_{\la-\dt}=c_2 x_{2}+\ldots c_{n-2} x_{n-2}+c_{n-1}x_{n-1}+c_n x_n,
\quad
$$
with  the scalar coefficients $c_i$ set to be
$$
c_i= (-1)^i q^{n-i-\frac{1}{2}}+(-1)^iq^{-(n-i-\frac{1}{2})},\quad 2\leqslant i\leqslant n,
$$
for  $N=2n+1$ and
$$
c_{i}=(-q)^{n-1-i}+(-q)^{-(n-1-i)},\quad 2\leqslant i\leqslant n-2,\quad c_{n-1}=c_n=1,
$$
for $N=2n$.
The $n-1$ coefficients $c_i$ satisfy the recurrent system of $n-2$ equations
\be
\begin{array}{ccc}
c_{i-1}+ac_{i}+c_{i+1}= c_{n-1}+c_{n}=0, \> N=2n+1,
\\
c_{i-1}+ac_{i}+c_{i+1}=c_{n-3}+a c_{n-2}+c_{n-1}+c_n= c_{n-2}+a c_{n-1}= c_{n-2}+a c_{n}=0, \> N=2n,
\end{array}
\label{rec_sys_c}
\ee
where $i$ varies from $3$ to $n-1$ in the first line  and from $3$ to $n-3$ in the second.
This system determines $(c_i)_{i=2}^n$ up to a common multiplier.

\begin{lemma}
Up to a scalar factor, $v_{\la-\dt}\in \hat M_\la$ is a unique vector of weight $\la-\dt$ annihilated by  $e_{\al}$, $\al\in \Pi^+-\{\al_2\}$.
\label{almost singular}
\end{lemma}
\begin{proof}
First of all, $v_{\la-\dt}$ is annihilated by $e_{\al_1}$, due to Lemma \ref{al1}.
Further proof is based on Proposition \ref{e_x}, stating that $e_{\al_i}v_{\la-\dt}=E_i x'_{i}$,
for some scalars $E_i$, $i=2,\ldots, n$. This yields a system of $n-2$ equations  $E_i=0$,
which is  written down in (\ref{rec_sys_c}) for each parity of $N$.
The coefficients $(c_i)_{i=2}^n$ are determined uniquely, up to a common factor.
\end{proof}
Recall that a vector in a $U_q(\g)$-module  is called singular if
it is annihilated by $\g_+$. In a module with highest weight, singular vectors generate proper
submodules.
\begin{propn}
\label{singular}
Suppose that $\la$ satisfies the condition
$$
q^{2(\al_2,\la)}=
- q^{- P}.
$$
Then the vector $v_{\la-\dt}\in \hat M_\la$ is  singular. Up to a scalar factor, it is
a unique singular vector of weight $\la-\dt$ and it exists
only if $\la$ satisfies the above condition.
\end{propn}
\begin{proof}
By Corollary \ref{basis}, $v_{\la-\dt}\not =0$. In view of Lemma \ref{almost singular}, we  only need to satisfy the condition  $e_{\al_2}v_{\la-\dt}=0$.
From Corollary \ref{basis}, we get $e_{\al_2}v_{\la-\dt}=E_2x_{2}'$ for some  scalar $E_2$. Evaluating $E_2x_{2}'$ and equating
it to zero we get conditions  on  $\la$ for $v_{\la-\dt}$ to be singular.
If $N=5$, we find $q^{(\al_2,\la)}(1-q)=q^{-(\al_2,\la)}(1-q^{-1})$, which immediately gives required
$q^{2(\al_2,\la)}=-q^{-1}$. For for $N=8$ we obtain
$$
  q^{(\al_2,\la)}q^2 c_2+q^{(\al_2,\la)}qc_{3} +  q^{(\al_2,\la)}qc_{4}
=  q^{-(\al_2,\la)}q^{-2}c_2+ q^{-(\al_2,\la)}q^{-1}c_{3}+  q^{-(\al_2,\la)}q^{-1}c_{4}.
$$
For $N=2n>8$ and $N=2n+1\geqslant 7$ we obtain
$$
q^{(\al_2,\la)}q^2c_2+
q^{(\al_2,\la)}qc_{3}
=q^{-(\al_2,\la)}q^{-2}c_2+
q^{-(\al_2,\la)}q^{-1}c_{3}.
$$
Plugging the expressions for $c_2$, $c_3$, $c_4$ in these equations we find that $v_{\la-\dt}$ is singular only if $\la$ satisfies the hypothesis. Its uniqueness follows from Lemma \ref{almost singular}.
\end{proof}

\subsection{The highest weight module $M_\la$ for general $\k$}
\label{module_M_general_k}
In this section we abandon the simplifying ansatz $\ell=0$, $m=2$ and allow for general isotropy subgroup $K$,
as in  (\ref{gr_K}).
The Lie algebra $\k$ of the subgroup $K$ and the maximal Levi subalgebra $\l$ read
\be
\k&=&\g\l(n_1)\oplus \ldots \oplus\g\l(n_\ell)\oplus \s\o(2m)\oplus\s\o(P),
\\
\l&=&\g\l(n_1)\oplus \ldots \oplus\g\l(n_\ell)\oplus \g\l(m)\oplus\s\o(P).
\label{Levi_l}
\ee
Consider the subalgebra $\g'=\g_{p+2}\subset \g$ with
the simple  positive roots $(\al_{n-p-1},\ldots,\al_n)$.
Under this embedding, the root $\al_2$ goes over to $\al_{n-p}$, and the
root $\delta$ reads
$$
\delta=\al_{n-p-1}+2\sum_{i=n-p}^{n-2}\al_i+\al_{n-1}+\al_n, \> \g=\s\o(2n),\quad
\delta=\al_{n-p-1}+2\sum_{i=n-p}^{n}\al_i, \> \g=\s\o(2n+1).
$$
Assuming $\la \in \mathfrak{C}^*_{\k}$, let $\hat M_\la$ be the parabolic Verma module over $U_q(\g)$.
Regarded as a $\g'$-weight by restriction,  $\la$ satisfies the
assumptions of Proposition \ref{singular}. Therefore, there is a singular vector $v_{\la-\dt}$ in the
$U_q(\g')$-submodule of $\hat M_\la$ generated by $v_\la$.
\begin{propn}
\label{singular_Verma}
Suppose that $\la \in \mathfrak{C}^*_{\k}$. Then $v_{\la-\dt}\in \hat M_\la$
is a unique, up to a scalar multiplier, singular vector of weight $\la-\dt$.
\end{propn}
\begin{proof}
The vector $v_{\la-\dt}$  is annihilated by $e_\al$ for $\al\in \Pi^+_{\g'}$, by to Proposition
\ref{singular}. Furthermore, $v_{\la-\dt}$ is constructed out of $f_\bt\in \Pi^+_{\g'}$, which
commute with $e_\al$ for $\al\in \Pi^+ -\Pi^+_{\g'}$. Consequently, such $e_\al$ annihilate $v_{\la-\dt}$ too.
Therefore $v_{\la-\dt}$ is singular in $\hat M_\la$.
 It is unique up to a factor as it is so for $U_q(\g')$.
\end{proof}
\begin{definition} Assuming $\la \in \mathfrak{C}^*_{\k}$, we denote by $\hat M_{\la-\dt}\subset \hat M_\la$
the submodule generated by $v_{\la-\dt}$ and we denote by $M_\la$ the quotient module  $\hat M_\la/\hat M_{\la-\dt}$.
\end{definition}
\noindent
The module  $M_\la$  is the key object of our approach to quantization and of our further study.

Next we prove that $M_\la$ is free as a $\C[\![\hbar]\!]$-module. There is a  PBW basis in $U_\hbar(\g_-)$
 generated by ordered quantum root vectors $f_\mu$, $\mu \in R^+$, which are defined through an action of the
quantum Weyl group on the generators, \cite{CP}. This basis establishes a natural $U_\hbar(\h)$-linear isomorphism of
$U_\hbar(\g_-)$ and $U(\g_-)\tp \C[\![\hbar]\!]$.
It is argued in \cite{M5} that, over the ring of scalars $\C[\![\hbar]\!]$,
one can arbitrarily change the ordering of $f_\mu$ and arbitrarily deform $f_\mu$ within $[U_\hbar(\g_-)]_\mu$.
Reordered deformed root vectors still generate a PBW-like basis. To apply this argument  to the orthogonal case,
we must find  an appropriate element $f_\dt\in [U_\hbar(\g_-)]_{-\dt}$.
\begin{lemma}
\label{q-root_vec}
There exists a deformation $f_\dt\in U_\hbar(\g_-)=U_q(\g_-)$ of the classical root vector of root $-\dt$ such that  $v_{\la-\dt}=f_\dt v_\la$.
\end{lemma}
\begin{proof}
Put $U_\hbar^-=U_\hbar(\g_-)$, $U^-=U(\g_-)$, and
define $\phi\in U_\hbar^-/ U_\hbar^-\l_-$ from the presentation $v_{\la-\dt}=\phi v_\la$. Note that $\phi$ is independent of
$\la$.
By Lemma \ref{almost singular},  $\phi\in   [U_\hbar^-/ U_\hbar^-\l_-]_{-\dt}$ is unique, up to a scalar multiplier,  solution of the system $[e_\al, \phi] =0 \mod U_\hbar^- \l$,  $\al \in \Pi^+_\l$.
Modulo $\hbar$, the classical root vector $f\in [\g]_{-\dt}$ solves this system. Indeed, it commutes with all such $e_\al$
as $\dt-\al$ is not a root once $\al \in \Pi^+_\l$. Therefore, upon a proper normalization,
the projection of $f$ to $U^-/ U^-\l_-$ coincides with the zero fiber of $\phi$. Regarding $f$
as an element of $U_\hbar^-=U_\hbar^-/U_\hbar^- \l_-\oplus U_\hbar^- \l_-$ under the natural linear isomorphism $U_\hbar^-\simeq U^-\tp \C[\![\hbar]\!]$, we
define its deformation $f_\dt$ by changing the $U_\hbar^-/ U_\hbar^-\l_-$-component to $\phi$. The
$U_\hbar^-\l_-$-component of $f$ can be replaced with its arbitrary deformation within $[U_\hbar^-\l_-]_{-\dt}$.
\end{proof}
Note that, in the symplectic case  \cite{M5}, the element $f_\dt$ participates in construction of $v_{\la-\dt}$.
In this exposition, we have introduced $v_{\la-\dt}$ in a simpler way, at the price of Lemma \ref{q-root_vec}.
\begin{propn}
The module $M_\la$ is free over $\C[\![\hbar]\!]$.
\end{propn}
\begin{proof}
The quantum "root vector" $f_\dt$ can be included in a PBW-like basis as discussed in \cite{M5}.
The rest of the proof is similar to the proof of Proposition 6.2 therein.
\end{proof}

We are going to  prove that the quantization of the conjugacy class $G/K$ can be realized by linear operators on
$M_\la$. To this  end, we study the module structure of the tensor product $\C^N\tp M_\la$ in the following section.
Again we start with the symmetric case $\ell=0$.
For technical reasons we process separately the cases of even and odd $N$.
\section{The $U_q(\g)$-module $\C^N\tp M_\la$ in the symmetric case}
In this section, we study the $U_q(\g)$-module  $\C^N\tp M_\la$.
To a large extent, the difference between Levi and non-Levi classes is concentrated in the
"symmetric part" of the stabilizer, so we consider this case first, as we did in the preceding sections.
We assume that the isotropy subalgebra $\k$ consists of two orthogonal blocks of rank $m$ and $p$,
$\k=\s\o(2m)\oplus \s\o(P)$,  where $P=2p$ for the $D$-series and $P=2p+1$ for
the $B$-series.

When restricted to the Levi subalgebra $\l=\g\l(m)\oplus \s\o(P)$, the natural $\g$-representation $\C^{N}$ splits into three irreducible
sub-representations, $\C^{N}=\C^{m}\oplus\C^{P}\oplus\C^{m} $. The submodule $\C^P$ carries the natural representation  of $\s\o(P)\subset \l$
while the two copies of $\C^m$ are the natural and conatural submodules of $\g\l(m)\subset \l$.
This reduction extends to the pair of the quantum groups $U_q(\l)\subset U_q(\g)$.

The natural and conatural $\g\l(m)$-submodules $\C^m$ glue up to the natural module of the block $\s\o(2m)\subset \k$ leading to the irreducible decomposition $\C^{N}=\C^{2m}\oplus\C^{P}$ over  $\k$. We cannot write a quantum version of this reduction because we do not know a natural candidate for the subalgebra $U_q(\k)\subset U_q(\g)$. Yet we bypass this obstacle.

We fix the standard basis  $\{w_i\}_{i=1}^{N}\subset \C^{N}$
of columns with the only non-zero entry in the $i$-th position.
The highest weights  of the irreducible $U_q(\l)$-submodules in $\C^N$ are $\ve_1$, $\ve_{m+1}$, $-\ve_{m}$, and
the corresponding weight vectors are $w_1$, $w_{m+1}$, $w_{N+1-m}$.
For generic $\la\in \mathfrak{C}^*_{\l}$, the tensor product $\C^N\tp \hat M_\la$ splits into the direct sum of three $U_q(\g)$-modules
 $\C^N\tp \hat M_\la=\hat M_1\oplus \hat M_2\oplus \hat M_3$, of highest weights $\nu_1=\la+\ve_1$, $\nu_2=\la+\ve_{m+1}$,  and $\nu_3=\la-\ve_{m}$, respectively.
 We shall prove, for almost all $\la \in \mathfrak{C}^*_{\k}$, the direct decomposition $\C^N\tp M_\la=M_1\oplus M_2$,
  where  $M_i$ are the images of $\hat M_i$  under the projection $\C^N\tp \hat M_\la\to \C^N\tp M_\la$. This
results in a degree reduction of the minimal polynomial for the quantum coordinate matrix (a similar effect
is produced on the classical coordinate matrix by the transition from $G/L$ to $G/K$).
This is the key step of our strategy
.

Let $u_{\nu_i}$, $i=1,2, 3$, denote the canonical generators of $\hat M_i\subset \hat M_\la$.
It can be shown that  $u_{\nu_2}\sim w_1\tp v_{\la-\dt}$ and vanishes in $\C^N\tp M_\la$, $\la\in \mathfrak{C}^*_{\k}$. The vector $u_{\nu_1}=w_1\tp v_\la$ carries the highest weight $\la+\ve_1$ in $\C^N\tp \hat M_\la$ and generates the submodule $\hat M_{1}$.
The singular vector of weight $\la+\ve_{m+1}$ that generates $\hat M_{2}$ reads
\be
u_{\nu_2}&=&\frac{q^{(\al,\la)}-q^{-(\al,\la)}}{q-q^{-1}}w_{m+1}\tp v_\la+(-q)^{-1}w_{m}\tp f_{\al_m} v_\la
+\ldots +(-q)^{-m} w_{1}\tp f_{\al_1}\ldots f_{\al_m} v_\la.
\nn
\ee
It is calculated in \cite{M5} for the symplectic case and still valid for orthogonal $\g$, because
it  involves only the generators of $\g\l(n)\subset \g$. Note that vector $u_{\nu_2}$ is also singular in $\C^N\tp M_\la$, as it is not nil there.
The following fact is established in
\cite{M5}.
\begin{lemma}
\label{gl-sym}
The singular vector $u_{\nu_2}$ is equal to
$q^{-m}\frac{q^{(\al,\la)+m}-q^{-(\al,\la)-m}}{q-q^{-1}} w_{m+1}\tp v_\la$ modulo
$\hat M_1$.
\end{lemma}

To proceed with the analysis of module structure of $\C^N\tp \hat M_\la$, we have to
develop a special diagram technique, which takes the rest of this section.
The action of a positive (negative) Chevalley generator on the standard basis $\{w_i\}_{i=1}^{N}$
features the following property: the line $\C w_i$ is either annihilated or mapped onto the line $\C w_k$ for some $k$.
 It is convenient to depict  such an action graphically.
Further we consider negative generators, since positive can be obtained by reversing
the arrows.

Up to an invertible scalar multiplier, the action of the family
$\{f_{\al}\}_{\al \in \Pi^+}\subset U_q\bigl(\s\o(2n+1)\bigr)$ on the standard basis $\{w_i\}_{i=1}^{2n+1}$ in $\C^{2n+1}$
is encoded in the following scheme:

\begin{center}
\begin{picture}(390,45)
\put(0,0){$w_{2n+1}$}
\put(10,15){\circle{3}}
\put(45,15){\vector(-1,0){30}}
\put(50,15){\circle{3}}
\put(100,15){$\ldots$}
\put(85,15){\vector(-1,0){30}}
\put(45,0){$w_{2n}$}
\put(160,15){\circle{3}}
\put(155,15){\vector(-1,0){30}}
\put(150,0){$w_{n+2}$}
\put(195,15){\vector(-1,0){30}}
\put(200,15){\circle{3}}
\put(190,0){$w_{n+1}$}
\put(235,15){\vector(-1,0){30}}
\put(240,15){\circle{3}}
\put(230,0){$w_{n}$}
\put(275,15){\vector(-1,0){30}}
\put(345,15){\vector(-1,0){30}}
\put(350,15){\circle{3}}
\put(345,0){$w_{2}$}
\put(290,15){$\ldots$}
\put(385,15){\vector(-1,0){30}}
\put(390,15){\circle{3}}
\put(385,0){$w_{1}$}
\put(25,20){$f_{\al_1}$}
\put(65,20){$f_{\al_2}$}
\put(135,20){$f_{\al_{n-1}}$}
\put(175,20){$f_{\al_{n}}$}
\put(215,20){$f_{\al_{n}}$}
\put(255,20){$f_{\al_{n-1}}$}
\put(325,20){$f_{\al_{2}}$}
\put(365,20){$f_{\al_{1}}$}

\end{picture}
\end{center}
This diagram has simple linear structure without branching and cycles.
The diagram for $\g=\s\o(2n)$ is more complicated:
\begin{center}
\begin{picture}(440,60)
%\put(10.5,30){\line(1,0){37}}
\put(0,0){$w_{2n}$}
\put(10,15){\circle{3}}
\put(50,15){\circle{3}}
\put(45,15){\vector(-1,0){30}}
\put(85,15){\vector(-1,0){30}}
\put(100,15){$\ldots$}
\put(155,15){\vector(-1,0){30}}
\put(45,0){$w_{2n-1}$}
\put(160,15){\circle{3}}
\put(150,0){$w_{n+2}$}
\put(195,15){\vector(-1,0){30}}
\put(200,15){\circle{3}}
\put(190,0){$w_{n+1}$}
\qbezier(163,22)(200,48)(237,22) \put(162,21.5){\vector(-2,-1){2}}
\qbezier(203,22)(240,48)(277,22) \put(202,21.5){\vector(-2,-1){2}}
\put(240,15){\circle{3}}
\put(230,0){$w_{n}$}
\put(275,15){\vector(-1,0){30}}
\put(280,15){\circle{3}}
\put(270,0){$w_{n-1}$}
\put(330,15){$\ldots$}
\put(315,15){\vector(-1,0){30}}
\put(385,15){\vector(-1,0){30}}
\put(390,15){\circle{3}}
\put(385,0){$w_{2}$}
\put(425,15){\vector(-1,0){30}}
\put(430,15){\circle{3}}
\put(425,0){$w_{1}$}
\put(25,20){$f_{\al_1}$}
\put(65,20){$f_{\al_2}$}
\put(135,20){$f_{\al_{n-2}}$}
\put(172,20){$f_{\al_{n-1}}$}
\put(195,38){$f_{\al_{n}}$}
\put(252,20){$f_{\al_{n-1}}$}
\put(290,20){$f_{\al_{n-2}}$}
\put(235,38){$f_{\al_{n}}$}
\put(365,20){$f_{\al_{2}}$}
\put(405,20){$f_{\al_{1}}$}
\end{picture}
\end{center}
The arrows designate the action up to a non-zero scalar, which
is equal to $-1$ on the left part of the diagram and $+1$ on the right part.
More exactly, moving along the diagram from left to right, the first occurrence of
$f_\al$ produces $-1$, while the second gives $+1$. In the matrix language, the non-zero entries above and below the skew diagonal are
$+1$ and $-1$, respectively. Along with this sign rule, the above graphs for $\{f_\al\}$ and $\{e_\al\}$ determine the
representation of $U_q(\g)$ on $\C^N$.

Further we adopt the following convention. If a vector $v$ is proportional to a vector $u$ with a scalar coefficient
$c\not =0$, i.e. $v=cu$, we write $v\simeq u$  and say that $v$ is equivalent to $u$.
If the difference $v-cu$ belongs to a vector space $W$, we write $v\simeq u\!\!\mod W$.

To study the action of $\{f_{\al}\}_{\al \in \Pi^+}$ on $\C^N\tp \hat M_\la$, we develop our diagram language further.
First of all, we transpose the above diagrams of the natural representation to
columns, so the arrows become vertical and oriented downward.
Suppose $(v^i)_{i=1}^l\in \hat M_\la$ is a finite sequence of vectors.
We associate a horizontal graph with nodes $(v^i)_{i=1}^l$ and arrows designating the action of $\{f_{\al}\}_{\al\in \Pi^+}$ on $\hat M_\la$,
in a similar fashion as vertical but with the following difference: it involves {\em not all} possible arrows $v^k\leftarrow v^i$
but only those of our interest. We still assume that the chosen arrows are isomorphisms of lines
spanned by $v^i$.
This implies that, up to a non-zero scalar factor, the nodes are {\em determined} by the subset of
maximal nodes (having no inward arrows) and by the set of arrows. In our case, there will be only one maximal vector $v^1=v_\la$,
hence the other nodes are determined by arrows. This implies that the horizontal graph is connected.

Let $\mathrm{Arr}(w_k)$ denote the set of negative Chevalley generators  whose arrows are directed from $w_k$.
Similarly, $\mathrm{Arr}(v^i)$ denote the set of generators whose arrows are directed from $v^i$.
For instance, $\mathrm{Arr}(w_k)$  consists of only one element for $k<2n+1$ and is empty for $k=2n+1$, in the series $B$.
We say that arrow $f$ has length $k$ if it sends node $i$ to the node $i+k$. All vertical arrows but $f_{\al_n}$ for
even $N$ have length $1$.

We construct a "tensor product" of vertical and horizontal graphs to be a diagram
with the nodes $w_k\tp v^i$, $k=1,\ldots, N$, $i=1,\ldots, l$. The factor $w_k$ marks the rows from top to bottom,
while $v^i$ marks the columns from  right to left.
The vertical and horizontal arrows designate the action of the designated Chevalley generators on the {\em tensor factors},
up to a scalar multiplier.
Under these assumptions, such diagrams  provide information about the action of $\{f_{\al}\}_{\al\in \Pi^+}$
not just on the tensor factors
but on the entire tensor product $\C^N\tp \hat M_\la$, in the following sense:
\begin{propn}
Suppose a horizontal arrow designates the action of  a Chevalley generator $f_{\al}$ on $v^i$.
If $f_{\al}\not \in \mathrm{Arr}(w_k)$, then $f_{\al}(w_k\tp v^i)\simeq  w_k\tp (f_{\al}v^i)$, otherwise
$f_{\al}(w_k\tp v^i)\simeq   w_k\tp f_{\al}v^i$ modulo $\C f_{\al}w_k\tp v^i$.
\end{propn}
\begin{proof}
This statement follows from the definition of diagram and quasi-primitivity of
the Chevalley generators (cf. the comultiplication in Section \ref{QOG}). In particular,
$f_\al\not \in \mathrm{Arr}(w_k)$  if and only if $f_\al w_k=0$, hence the first alternative. The second alternative is also an immediate consequence of quasi-primitivity of
$f_\al$.
\end{proof}
\begin{corollary}
\label{interval}
Suppose that  $f_\al\in \mathrm{Arr}(v^i) - \mathrm{Arr}(w_k)$ for some $i$ and $k$. Suppose that
 for $k_1\leqslant j\leqslant k$
the nodes $w_j\tp v^i$ lie in a submodule $\hat M\subset \C^N\tp \hat M_\la$
 and all arrows from $\mathrm{Arr}(w_j)$ have length $1$.
Then  $\Span \{w_j\tp f_\al v^i\}_{j=k_1}^k\subset \hat M$.
\end{corollary}
\begin{proof}
We have $f_\al(w_k\tp v^i)\simeq w_k\tp f_\al v^i$, as $f_\al w_k=0$. Therefore,
$\Span\{f_\al(w_j\tp v^i)\}_{j=k_1}^k=\Span\{w_j\tp f_\al v^i\}_{j=k_1}^k$ modulo $\Span\{w_j\tp v^i\}_{j=k_1}^k\subset \hat M$.
Now the proof is immediate.
\end{proof}

Suppose there are intervals of vertical nodes $(w_k)$, $k\in I_v=[k_1,k_2]$, and horizontal nodes $(v^i)$, $i\in I_h=[i_1,i_2]$,
such that: {\em all} vertical arrows directed from $w_k$, $k\in I_v'=[k_1,k_2-1]$ are  of length $1$;
for each $i\in I_h'=[i_1,i_2-1]$  {\em there is} a horizontal arrow of length $1$. In particular,  $\mathrm{Arr}(w_k)$ consists of one element for all $k\in I_v'$.
 Let us  denote the subset of
these selected horizontal arrows by $A_h$. Consider the subgraph with nodes $(w_k\tp v^i)$, $(k,i)\in I_v\times I_h$,
the vertical arrows from  $\mathrm{Arr}(w_k)$, $k\in I_v'$, and the horizontal arrows from the selected subset $A_h$.
 We call such a subgraph {\em simple rectangle}. In particular, the entire diagram may be simple.
 Its  horizontal and vertical "tensor factor" subgraphs are topologically  simply connected, having no cycle or branching.

The diagrams of interest will be specified in the next section. Here we establish a general fact, which will be used
in what follows.
\begin{lemma}
\label{triangle}
Suppose that a diagram $D$ contains an equilateral rectangular triangle $T$ (leveled by top and right edges) belonging to a simple
rectangle in $D$. Suppose that
the right edge of $T$ belongs to a submodule $\hat M\subset \C^N\tp \hat M_\la$. Then the entire triangle $T$ belongs
to $\hat M$.
\end{lemma}
\begin{proof}
Without loss of generality we may assume that $D$ is simple and $T$ sits in the north-east corner of $D$,
i.e. $w_1\tp v^1\in \hat M$ is its maximal node.
Suppose that its edge contains $t$ nodes.
We do induction on column's number $i$, which is illustrated below (on the left).
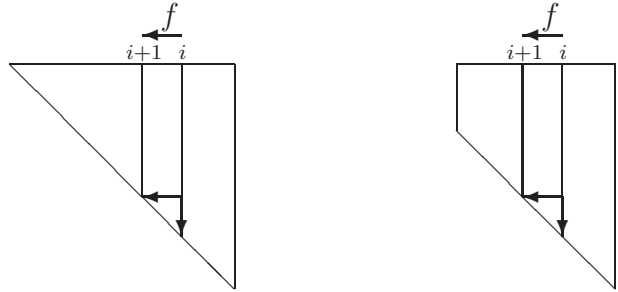
\begin{center}
\begin{picture}(100,100)
\put(72,100){$f$}\put(59,87){$\scriptstyle{i+1}$}\put(79,87){$\scriptstyle{i}$}
\put(100,0 ){\line(0,1){85}}\put(15,85){\line(1,0){85}}
\thinlines\put(15,85){\line(1,-1){85}}
\put(80,85){\line(0,-1){65}}\put(65,85){\line(0,-1){50}}
\thicklines\put(80,35){\vector(0,-1){15}}\thicklines\put(80,35){\vector(-1,0){15}}
\thicklines\put(80,96){\vector(-1,0){15}}
\end{picture}
\hspace{0.5in}
\begin{picture}(100,110)
\put(72,100){$f$}\put(59,87){$\scriptstyle{i+1}$}\put(79,87){$\scriptstyle{i}$}
\put(100,0 ){\line(0,1){85}}\put(40,85){\line(1,0){60}}\put(40,85){\line(0,-1){25}}
\thinlines\put(40,60){\line(1,-1){60}}
\put(80,85){\line(0,-1){65}}\put(65,85){\line(0,-1){50}}
\thicklines\put(80,35){\vector(0,-1){15}}\thicklines\put(80,35){\vector(-1,0){15}}
\thicklines\put(80,96){\vector(-1,0){15}}
\end{picture}
\end{center}
By the hypothesis, the column $\{w_k\tp v^1\}_{k=1}^t$ lies in $\hat M$.
Suppose that the statement is proved for $1\leqslant i< t$. Let $f\in U_q(\g)$ be the operator assigned
to the horizontal arrow $v^{i+1}\leftarrow v^i$.
Each node $w_k\tp v^i$, $k=1,\ldots, t-i,$
is sent by $f$ to $w_k\tp v^{i+1}$ possibly modulo $\C w_{k+1}\tp v^{i}$,  which belongs to $\hat M$
by the induction assumption. Therefore, the $i+1$-st column of $T$ does belong to $\hat M$.
\end{proof}
Remark that the triangle can be replaced with a trapezoid obtained by cutting off the left end of $T$ with vertical line, as shown
on the right.

The sequence  $(v^i)$ is assumed to be finite and contain the unique minimal node (no outward arrows).
This node is in the focus of our interest. By construction, it will carry the
weight $\la-\dt$, and the whole diagram yield a path (paths) to it from the maximal node $v_\la$.
This way, we associate a diagram with every non-vanishing
Chevalley monomial in $[\hat M_\la]_{\la-\dt}$ participating in the singular vector $v_{\la-\dt}$.

\subsection{Series $B$, symmetric case $\k=\s\o(4)\oplus \s\o(2n-3)$}
Suppose first that $m=2$. Later on we drop this restriction.
Given a permutation $s$ of $1,\ldots, n$, we define Chevalley monomials $v_s^{i}\in \hat M_\la$, $i=1,\ldots ,2n$, through the graph
\begin{center}
\begin{picture}(390,30)
\put(50,15){\circle{3}}
\put(100,15){$\ldots$}
\put(85,15){\vector(-1,0){30}}
\put(45,0){$v_s^{2n}$}
\put(160,15){\circle{3}}
\put(155,15){\vector(-1,0){30}}
\put(150,0){$v_s^{n+1}$}
\put(195,15){\vector(-1,0){30}}
\put(200,15){\circle{3}}
\put(195,0){$v_s^{n}$}
\put(235,15){\vector(-1,0){30}}
\put(240,15){$\ldots$}
\put(290,15){\vector(-1,0){30}}
\put(295,15){\circle{3}}
\put(330,15){\vector(-1,0){30}}
\put(335,15){\circle{3}}
\put(320,0){$v_s^{1}=v_\la$}

\put(60,22){$f_{\al_{s(1)}}$}
\put(128,22){$f_{\al_{s(n-1)}}$}
\put(170,22){$f_{\al_{s(n)}}$}
\put(215,22){$f_{\al_{n}}$}
\put(270,22){$f_{\al_{3}}$}
\put(310,22){$f_{\al_{2}}$}
\end{picture}
\end{center}
In particular, the minimal node of the graph is $v^{2n}_s=f_{\al_{s(1)}}\ldots f_{\al_{s(n)}}f_{\al_n} \ldots f_{\al_2}v_\la\in [\hat M_{\la}]_{\la-\dt}$.
For $s=\id$ we omit the subscript $s$ and denote $v^{i}_s$ simply by $v^{i}$.

With a permutation $s$ of $1,\ldots, n$  such that $v_s^{2n}\not =0$ we associate a diagram $D_s$, as explained in the preceding
section. Essential
are the nodes $\{w_k\tp v^i\}$ with $i+k\leqslant 2n+1$, so we display only this triangular part of $D_s$:
\begin{center}
\begin{picture}(350,180)
\put(170,20){$D_s$}
\put(5,160){$\scriptstyle{w_1\tp v_s^{2n}}$}
\put(120,160){$\scriptstyle{w_1\tp v_s^{n+1}}$}\put(185,160){$\scriptstyle{w_1\tp v^{n}}$}
\put(300,160){$\scriptstyle{w_1\tp v^1}$}

\put(65,162 ){\vector(-1,0){30}}\put(70,162 ){$\ldots$}\put(118,162 ){\vector(-1,0){30}}
\put(180,161 ){\vector(-1,0){30}}
\put(240,161 ){\vector(-1,0){30}}\put(242,161 ){$\ldots$}\put(290,161 ){\vector(-1,0){30}}

\put(133,157 ){\vector(0,-1){18}}\put(198,157){\vector(0,-1){18}}\put(315,157 ){\vector(0,-1){18}}
\put(131,125 ){$\vdots$}\put(196,125 ){$\vdots$}\put(313,125 ){$\vdots$}

\put(85,125){\circle{1}}\put(75,130){\circle{1}}\put(65,135){\circle{1}}

\put(133,122 ){\vector(0,-1){18}}

\put(198,122 ){\vector(0,-1){18}}\put(315,122 ){\vector(0,-1){18}}
\put(180,100 ){\vector(-1,0){30}}\put(240,100 ){\vector(-1,0){30}}\put(242,100 ){$\ldots$}\put(290,100 ){\vector(-1,0){30}}
\put(118,98){$\scriptstyle{w_{n}\tp v_s^{n+1}}$}\put(183,98){$\scriptstyle{w_{n}\tp v^{n}}$}\put(300,98){$\scriptstyle{w_{n}\tp v^1}$}

\put(198,95 ){\vector(0,-1){18}}\put(315,95 ){\vector(0,-1){18}}
\put(175,70){$\scriptstyle{w_{n+1}\tp v^{n}}$}\put(291,70){$\scriptstyle{w_{n+1}\tp v^1}$}
\put(240,71 ){\vector(-1,0){30}}\put(242,71 ){$\ldots$}\put(290,71 ){\vector(-1,0){30}}

\put(315,68 ){\vector(0,-1){18}}
\put(313,39 ){$\vdots$}
\put(315,37 ){\vector(0,-1){18}}
\put(297,10){$\scriptstyle{w_{2n}\tp v^1}$}

\put(43,169){$\scriptstyle{f_{\al_{s(1)}}}$}\put(91,169){$\scriptstyle{f_{\al_{s(n-1)}}}$}
\put(157, 169){$\scriptstyle{f_{\al_{s(n)}}}$}\put(220,169){$\scriptstyle{f_{\al_n}}$}\put(270,169){$\scriptstyle{f_{\al_2}}$}

\put(320,27){$\scriptstyle{f_{\al_{2} }}$}\put(320,88){$\scriptstyle{f_{\al_{n}}}$}\put(320,58){$\scriptstyle{f_{\al_{n}}}$}
\put(320,118){$\scriptstyle{f_{\al_{n-1}}}$}\put(320,148){$\scriptstyle{f_{\al_1}}$}

\put(260,38){\circle{1}}\put(250,43){\circle{1}}\put(240,48){\circle{1}}

\end{picture}
\end{center}
Note that the first $n$ columns (counting from the right) in all $D_s$ are the same.

Denote by $D_\id'\subset D_\id$  the sub-graph above the principal diagonal, i.e. consisting of nodes
$\{w_k\tp v^j\}$ such that $k+j\leqslant 2n$.
Given $s\not =\id$ let $i$ be the highest of $1,\ldots,n$ displaced by $s$, i.e. $s(i)\not=i$.
Denote  by $D_s'\subset D_s$  the trapezoid of nodes  $\{w_k\tp v^j_s\}$ obeying
 $k+j\leqslant 2n+1$, $k\leqslant i$.
\begin{center}
\begin{picture}(420,100)
\put(190,0){$D_s',\quad s\not= \id$}
\put(0,80){$\scriptstyle{w_1\tp v_s^{2n}}$}\put(55,80){\vector(-1,0){30}}\put(57,80){\ldots}
\put(102,80){\vector(-1,0){30}}
\put(103,80){$\scriptstyle{w_1\tp v_s^{2n-i+1}}$}
\put(158,80){\vector(-1,0){30}}
\put(160,80){$\scriptstyle{w_1\tp v_s^{2n-i}}$}
\put(225,80){\vector(-1,0){30}}\put(227,80){\ldots}\put(272,80){\vector(-1,0){30}}
\put(275,80){$\scriptstyle{w_1\tp v_s^{n}}$}
\put(335,80){\vector(-1,0){30}}\put(337,80){\ldots}\put(382,80){\vector(-1,0){30}}
\put(390,80){$\scriptstyle{w_1\tp v_\la}$}
\put(403,78 ){\vector(0,-1){20}}
\put(401,47 ){$\vdots$}
\put(403,45 ){\vector(0,-1){20}}

\put(288,78 ){\vector(0,-1){20}}
\put(286,47 ){$\vdots$}
\put(288,45 ){\vector(0,-1){20}}

\put(173,78 ){\vector(0,-1){20}}
\put(171,47 ){$\vdots$}
\put(173,45 ){\vector(0,-1){20}}

\put(115,78 ){\vector(0,-1){20}}
\put(113,47 ){$\vdots$}
\put(115,45 ){\vector(0,-1){20}}

\put(75,45 ){\circle{1}}\put(65,50 ){\circle{1}}\put(55,55 ){\circle{1}}

\put(160,18){$\scriptstyle{w_i\tp v_s^{2n-i}}$}
\put(103,18){$\scriptstyle{w_i\tp v_s^{2n-i+1}}$}
\put(158,20){\vector(-1,0){30}}
\put(225,20){\vector(-1,0){30}}\put(227,20){\ldots}\put(272,20){\vector(-1,0){30}}
\put(275,18){$\scriptstyle{w_i\tp v_s^{n}}$}
\put(335,20){\vector(-1,0){30}}\put(337,20){\ldots}\put(382,20){\vector(-1,0){30}}
\put(390,18){$\scriptstyle{w_i\tp v_\la}$}
\put(32,90){$\scriptstyle{f_{\al_{s(1)}}}$}\put(135,93){$\scriptstyle{f_{\al_{s(i)}}}$}
\put(204,90){$\scriptstyle{f_{\al_{i+1}}}$}\put(255,90){$\scriptstyle{f_{\al_{n}}}$}\put(315,90){$\scriptstyle{f_{\al_{n}}}$}
\put(363,90){$\scriptstyle{f_{\al_{2}}}$}
\put(407,68){$\scriptstyle{f_{\al_{1}}}$}
\put(407,32){$\scriptstyle{f_{\al_{i-1}}}$}
\end{picture}
\end{center}

Let $\hat M$ denote the $U_q(\g)$-submodule $\hat M_1+\hat M_2\subset \C^N\tp \hat M_\la$.
\begin{lemma}
Suppose that $q^{3-2p}\not =-1$. Then $D_s'$ lies in $\hat M$.
\label{sigma=id}
\end{lemma}
\begin{proof}
First suppose that $s=\id$.
The statement is trivial for the node $u_{\nu_1}=w_1\tp v_\la$ generating $\hat M_1$.
By Lemma \ref{gl-sym}, $u_{\nu_2}\simeq w_2\tp v_\la \!\!\mod\hat M_1$, under the assumption $q^{3-2p}\not =-1$. Hence the $w_2\tp v_\la$
belongs to $\hat M$ too. Further observe that the rightmost column of $D_\id'$ lies in $\hat M$.
For this part of $D_\id'$, the vertical arrows actually depict the action on the whole tensor square, as they kill the factor $v_\la$.
Notice that  $D_s$ is simple. One is left to apply Lemma \ref{triangle} to the triangle $T=D'_\id$.

Now we consider the case $s\not =\id$ and let $i$ be the highest integer displaced by $s$.
Observe that the right rectangular part of $D'_\id$ up to column $2n-i$ is
the same as in $D_\id$ and belongs to $D_\id'$. Hence it lies in  $\hat M$, as already proved.
Since $s(i)\not =i$, the horizontal arrow $f_{\al_{s(i)}}\in \mathrm{Arr}(v^{2n-i}_s)$ is distinct from
vertical $f_{\al_{i}}$ constituting $ \mathrm{Arr}(w_i)$ (suppressed in the graph).
By Corollary \ref{interval}, column $2n-i+1$ of $D_\id'$ belongs to  $\hat M$.
The remaining part of $D'_s$ is a triangle bounded on the right
with column $2n-i+1$. By Lemma \ref{triangle}, it belongs to $\hat M$.
\end{proof}

Now we are ready to prove the following
\begin{propn}
\label{dir_sum_symB}
Suppose that $q^{2m-2p-1}\not =-1$. Then the tensor product $\C^N\tp M_\la$ splits into the direct sum $M_1\oplus M_2$.
\end{propn}
\begin{proof}
We will use an operator $\Q\in \End(\C^N\tp \hat M_\la)$ defined by (\ref{Q-matrix}). Here we need to know
that $\Q$ is $U_\hbar(\g)$-invariant and turns scalar on $\hat M_1$ and $\hat M_2$ with the eigenvalues,  respectively, $\mu_1=-q^{-2p-1}$ and $\mu_2=q^{-2m}$, cf. (\ref{-1+1-1}). All eigenvalues of $\Q$ are calculated in \cite{M2} and presented in (\ref{e_v in hatM_la}). Since $\mu_1\not =\mu_2$, the modules $M_1$ and $M_2$ have zero intersection,
and their sum is direct.

We  prove the statement if we show that $\C^N\tp v_\la\subset M=M_1\oplus M_2$.
We consider the case $m=2$ first. Then $w_{2n}$ is the highest weight vector of the $\l$-submodule in $\C^N$ of weight $-\ve_2$.
Our strategy is to show that the vertical scheme representing the $U_q(\b^-)$-action on $\C^N$
yields the action on $\C^N\tp v_\la$ modulo $M$. Starting from $w_{1}\tp v_\la\in M$ we obtain all
$w_{i}\tp v_\la \mod M$ by applying the Chevalley generators. The hardest part of the job  is
the transition from $w_{2n-1}\tp v_\la$ to $w_{2n}\tp v_\la$.

By Lemma \ref{sigma=id}, the diagonal of $D_\id$ right over the principal diagonal lies in $\hat M$. Notice that
the vertical and horizontal arrows applied to all nodes in this diagonal coincide. Therefore, up to a non-zero scalar
factor,  the elements on the principal diagonal are all equivalent modulo $\hat M$. For instance, apply $f_{\al_1}$ to
$w_1\tp v^{2n-1}\in \hat M$ and get $q^{2}w_2\tp v^{2n-1}+w_1\tp v^{2n}\in \hat M$,
hence $w_1\tp v^{2n}= -q^{2}w_2\tp v^{2n-1}\!\!\mod\hat M$. Moving further down the
principal diagonal, we find $w_1\tp v^{2n}\simeq w_{2n}\tp v_\la \!\!\mod \hat M$.
Now notice that $x_{2}=v^{2n}-av^{2n}_s$, where $s$ is the transposition $(1,2)$.
Observe that all other $x_{i}$ are linear combinations
of $v^{2n}_s$ for certain $s \not=\id$. Since $v^{2n}_s \in \hat M$ once $s\not =\id$, by Lemma \ref{sigma=id},
$w_1\tp v_{\la-\dt}\simeq w_{2n}\tp v_\la$ modulo $\hat M$.
Hence $w_{2n}\tp v_\la \in M$. Applying $f_{\al_1}$ to $w_{2n}\tp v_\la$ we get $w_{2n+1}\tp v_\la\in M$, up to a scalar factor.

We have proved the inclusion $\C^N\tp v_\la\subset  M$ under the assumption $m=2$. Now we drop this restriction. First of
all, $w_1\tp v_\la\in M$ and $w_i\tp v_\la=f_{\al_{i-1}}(w_{i-1}\tp v_\la)\in M $ for $i=2,\ldots, m$.
The transition from $w_m\tp v_\la$ to $w_{m+1}\tp v_\la$ is facilitated by Lemma \ref{gl-sym}
and is similar to the case $m=2$. Namely,
$w_{m+1}\tp v_\la \simeq u_{\nu_2}\in M_2\subset M$ modulo $M$ under the assumption $q^{2m-2p-1}\not =-1$.
This is reducing the proof to the case $m=2$. Recall $\g'=\g_{p+2}\subset \g$ defined in Section \ref{module_M_general_k}. Acting by $\{f_\al\}_{\al}$, $\al \in \Pi^+_{\g'}$,
 on $w_{m-1}\tp v_\la$ we proceed as in the $m=2$-case and check that
$w_{i}\tp v_\la\in M$, $i=m+2,\ldots,N-m+2$. Applying
 $\l_-$ to $w_{N+2-m}\tp v_\la\in M$ we get $w_{i}\tp v_\la\in M$ for $i=N-m+3,\ldots,N$,
 as required. The inclusion $\C^N\tp M_\la \subset M$ is proved.
\end{proof}

\subsection{Series $D$, symmetric case $\k=\s\o(4)\oplus\s\o(2n-4)$}
We have to consider two types of diagrams for $\g=\s\o(2n)$ associated with
Chevalley monomials constituting the vectors  $x_{i}$, $i=2,\ldots, n-2$ on the one hand and the
 "tail" vectors $x_{n-1}$ and $x_{n}$, on the other.

Let us start with the first type.
Given a permutation  $s$ of $(1,\ldots, n-2)$ we define vectors
$v_s^i\in \hat M_\la$, $i=1,\ldots, 2n-1$,  through the graph
\begin{center}
\begin{picture}(440,45)
%\put(10.5,30){\line(1,0){37}}
\put(0,0){$v^{2n-1}_s$}
\put(10,15){\circle{3}}
\put(50,15){\circle{3}}
\put(45,15){\vector(-1,0){30}}
\put(85,15){\vector(-1,0){30}}
\put(100,15){$\ldots$}
\put(155,15){\vector(-1,0){30}}
\put(45,0){$v_s^{2n-2}$}
\put(160,15){\circle{3}}
\put(150,0){$v_s^{n+1}$}
\put(195,15){\vector(-1,0){30}}
\put(200,15){\circle{3}}
\put(190,0){$v_s^{n}$}
\qbezier(163,22)(200,48)(237,22) \put(162,21.5){\vector(-2,-1){2}}
\qbezier(203,22)(240,48)(277,22) \put(202,21.5){\vector(-2,-1){2}}
\put(240,15){\circle{3}}
\put(230,0){$v_s^{n-1}$}
\put(275,15){\vector(-1,0){30}}
\put(280,15){\circle{3}}
\put(270,0){$v_s^{n-2}$}
\put(330,15){$\ldots$}
\put(315,15){\vector(-1,0){30}}
\put(385,15){\vector(-1,0){30}}
\put(390,15){\circle{3}}
\put(385,0){$v_s^{2}$}
\put(425,15){\vector(-1,0){30}}
\put(430,15){\circle{3}}
\put(410,0){$v_s^{1}=v_\la$}
\put(25,22){$f_{\al_{s(1)}}$}
\put(65,22){$f_{\al_{s(2)}}$}
\put(128,22){$f_{\al_{{s(n-2)}}}$}
\put(172,22){$f_{\al_{n-1}}$}
\put(195,39){$f_{\al_{n}}$}
\put(252,22){$f_{\al_{n-1}}$}
\put(290,22){$f_{\al_{n-2}}$}
\put(235,39){$f_{\al_{n}}$}
\put(365,22){$f_{\al_{3}}$}
\put(405,22){$f_{\al_{2}}$}
\end{picture}
\end{center}
The minimal element of this sequence is
$v_s^{2n-1}=f_{\al_{s(1)}}\ldots f_{\al_{s(n-2)}}f_{\al_n} \ldots f_{\al_2}v_\la$.
The first $n+1$ nodes are independent of $s$. As for odd $N$, we drop the subscript $s$ from   $v^i_s$ for $s=\id$.

With every permutation $s$ such that $v^{2n-1}_s\not =0$ we associate the diagram $D_s$.
As before, we restrict consideration to the  triangular part of it, retaining only $w_k\tp v^i_s$ with $k+i\leqslant 2n$.
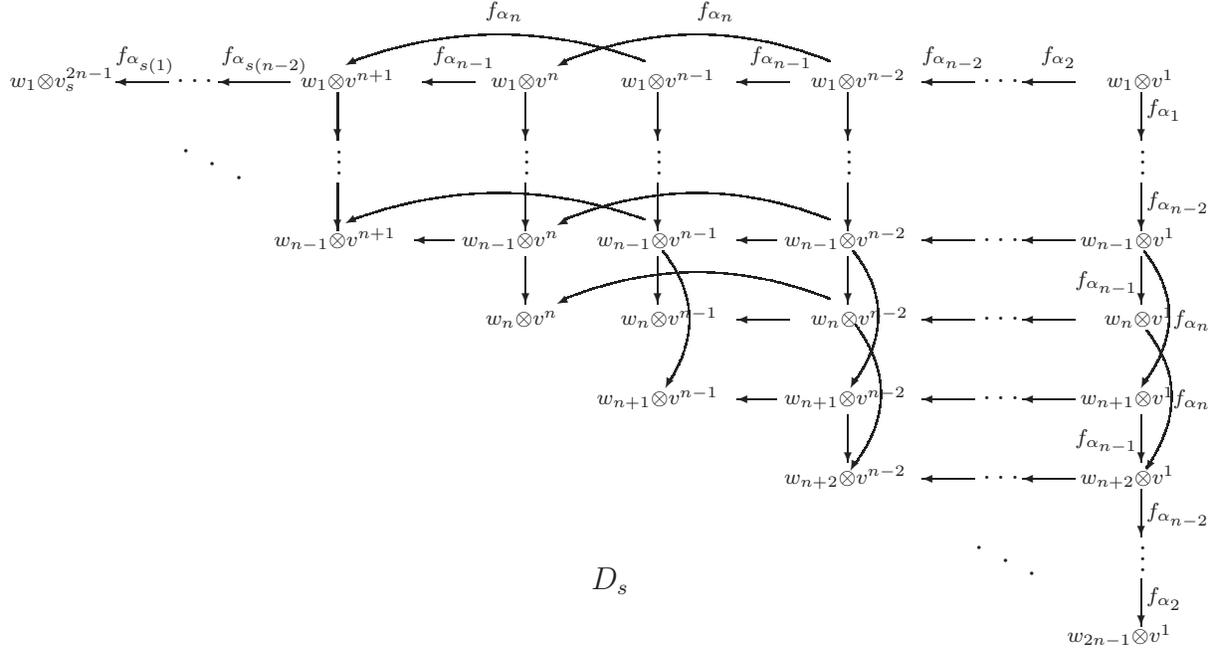
\begin{center}
\begin{picture}(450,250)
\qbezier(210,220)(260,240)(310,220) \put(209,219){\vector(-2,-1){2}}
\qbezier(130,220)(185,240)(240,220) \put(129,219.5){\vector(-2,-1){2}}

\qbezier(210,160)(260,180)(310,160) \put(209,159){\vector(-2,-1){2}}
\qbezier(130,160)(185,180)(240,160) \put(129,159.5){\vector(-2,-1){2}}

\qbezier(210,130)(260,150)(310,130) \put(209,129){\vector(-2,-1){2}}

\put(0,210){$\scriptstyle{w_1\tp v_s^{2n-1}}$}
\put(60,212 ){\vector(-1,0){20}}\put(63,212 ){$\ldots$}\put(106,212 ){\vector(-1,0){27}}
\put(110,210){$\scriptstyle{w_1\tp v^{n+1}}$}
\put(176,212 ){\vector(-1,0){20}}
\put(182,210){$\scriptstyle{w_1\tp v^{n}}$}
\put(231,210){$\scriptstyle{w_1\tp v^{n-1}}$}
\put(295,212 ){\vector(-1,0){20}}
\put(303,210){$\scriptstyle{w_1\tp v^{n-2}}$}
\put(365,212 ){\vector(-1,0){20}}
\put(368,212 ){$\ldots$}
\put(403,212 ){\vector(-1,0){20}}
\put(415,210){$\scriptstyle{w_1\tp v^1}$}

\put(65,183){$\cdot$}\put(75,178){$\cdot$}\put(85,173){$\cdot$}
\put(365,33){$\cdot$}\put(375,28){$\cdot$}\put(385,23){$\cdot$}

\put(428,208){\vector(0,-1){18}}
\put(427,177){$\vdots$}
\put(428,174){\vector(0,-1){18}}
\put(317,208){\vector(0,-1){18}}
\put(316,177){$\vdots$}
\put(317,174){\vector(0,-1){18}}
\put(245,208){\vector(0,-1){18}}
\put(244,177){$\vdots$}
\put(245,174){\vector(0,-1){18}}
\put(195,208){\vector(0,-1){18}}
\put(194,177){$\vdots$}
\put(195,174){\vector(0,-1){18}}
\put(124,208){\vector(0,-1){18}}
\put(123,177){$\vdots$}
\put(124,174){\vector(0,-1){18}}

\put(100,150){$\scriptstyle{w_{n-1}\tp v^{n+1}}$}
\put(168,152 ){\vector(-1,0){15}}
\put(171,150){$\scriptstyle{w_{n-1}\tp v^{n}}$}
\put(222,150){$\scriptstyle{w_{n-1}\tp v^{n-1}}$}
\put(290,152 ){\vector(-1,0){15}}
\put(293,150){$\scriptstyle{w_{n-1}\tp v^{n-2}}$}

\put(365,152 ){\vector(-1,0){20}}
\put(368,152 ){$\ldots$}
\put(403,152 ){\vector(-1,0){20}}
\put(365,122 ){\vector(-1,0){20}}
\put(368,122 ){$\ldots$}
\put(403,122 ){\vector(-1,0){20}}
\put(365,92 ){\vector(-1,0){20}}
\put(368,92 ){$\ldots$}
\put(403,92 ){\vector(-1,0){20}}
\put(365,62 ){\vector(-1,0){20}}
\put(368,62 ){$\ldots$}
\put(403,62 ){\vector(-1,0){20}}

\put(180,120){$\scriptstyle{w_n\tp v^{n}}$}
\put(231,120){$\scriptstyle{w_n\tp v^{n-1}}$}
\put(295,122 ){\vector(-1,0){20}}
\put(303,120){$\scriptstyle{w_n\tp v^{n-2}}$}

\put(222,90){$\scriptstyle{w_{n+1}\tp v^{n-1}}$}
\put(290,92 ){\vector(-1,0){15}}
\put(293,90){$\scriptstyle{w_{n+1}\tp v^{n-2}}$}

\put(293,60){$\scriptstyle{w_{n+2}\tp v^{n-2}}$}

\put(405,150){$\scriptstyle{w_{n-1}\tp v^1}$}
\put(428,146){\vector(0,-1){18}}\put(317,146){\vector(0,-1){18}}\put(245,146){\vector(0,-1){18}}\put(195,146){\vector(0,-1){18}}
\put(414,120){$\scriptstyle{w_{n}\tp v^1}$}
\put(405,90){$\scriptstyle{w_{n+1}\tp v^1}$}
\put(428,86){\vector(0,-1){18}}\put(317,86){\vector(0,-1){18}}
\put(405,60){$\scriptstyle{w_{n+2}\tp v^1}$}

\qbezier(429,98)(448,124)(429,148)\put(430,100){\vector(-1,-2){2}}
\qbezier(430,66)(449,94)(430,118)\put(432,69){\vector(-1,-2){2}}

\qbezier(319,98)(338,124)(319,148)\put(320,100){\vector(-1,-2){2}}
\qbezier(320,68)(339,94)(320,120)\put(320,69){\vector(-1,-2){2}}

\qbezier(249,98)(266,124)(247,148)\put(250,100){\vector(-1,-2){2}}

\put(428,58){\vector(0,-1){18}}
\put(427,27){$\vdots$}
\put(428,24){\vector(0,-1){18}}
\put(400,0){$\scriptstyle{w_{2n-1}\tp v^1}$}
\put(430,200){$\scriptstyle{f_{\al_1}}$}
\put(430,165){$\scriptstyle{f_{\al_{n-2}}}$}
\put(403,135){$\scriptstyle{f_{\al_{n-1}}}$}
\put(440,120){$\scriptstyle{f_{\al_{n}}}$}
\put(440,90){$\scriptstyle{f_{\al_{n}}}$}
\put(403,75){$\scriptstyle{f_{\al_{n-1}}}$}
\put(430,47){$\scriptstyle{f_{\al_{n-2}}}$}
\put(430,15){$\scriptstyle{f_{\al_{2}}}$}
\put(180,237){$\scriptstyle{f_{\al_{n}}}$}\put(260,237){$\scriptstyle{f_{\al_{n}}}$}
\put(40,220){$\scriptstyle{f_{\al_{s(1)}}}$}\put(80,220){$\scriptstyle{f_{\al_{s(n-2)}}}$}\put(160,220){$\scriptstyle{f_{\al_{n-1}}}$}\put(280,220){$\scriptstyle{f_{\al_{n-1}}}$}\put(345,220){$\scriptstyle{f_{\al_{n-2}}}$}\put(390,220){$\scriptstyle{f_{\al_{2}}}$}
\put(220,20){$D_s$}
\end{picture}
\end{center}
Only the part of $D_s$ which lies to the left of column $n+1$ depends on $s$.
We have emphasized this by omitting the subscript  $s$ in the right part.

The diagrams $D_s$ account for Chevalley  monomials $v_s^{2n-1}$ participating in $\{x_{i}\}_{i=2}^{n-2}$.
The vectors $x_{n-1}$ and $x_{n}$ bring about
different diagrams. Due to the symmetry between $x_{n-1}$ and $x_{n}$ we will consider only  $x_{n}$.
Define the set $\{v^i\}_{i=1}^{2n-1}\subset M_\la$ as follows. The first $n-1$ vectors  are as
before: $v^1=v_{\la}$,  $v^i=f_{\al_i}v^{i-1}$, $i=2,\ldots, n-1$. The remaining $n$ vectors are set to be
$$
v^{n-1+k}=f_{\al_k}v^{n+k-2}, \quad k=1,\ldots, n.
$$
The arrows $v^{i-1}\leftarrow v^i$ are uniquely determined by the set of nodes $\{v^i\}$.
Given a permutation $s$ of $1,\ldots, n-2$ we define the set $\{v_s^i\}_{i=1}^{2n-1}$
by
$v^{i}_s=v^{i}$, $i=1,\ldots, n-1$ and
$$
v_s^{n-1+k}=f_{\al_{s(k)}}v_s^{n+k-2}, \quad k=1,\ldots, n.
$$
In fact, these vectors are zero for most $s$. Of all $s$ we  only need the transposition $(1,2)$.
This is sufficient for $x_{n}$, which comprises two Chevalley  monomials, due to the factor $[f_{\al_1},f_{\al_2}]_a$
 in it.
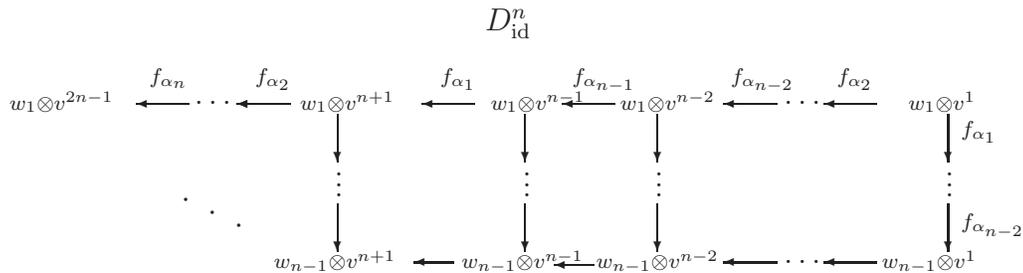
\begin{center}
\begin{picture}(370,100)
\put(180,90){${D_\id^n}$}

\put(0,60){$\scriptstyle{w_1\tp v^{2n-1}}$}
\put(68,62 ){\vector(-1,0){20}}
\put(70,62 ){$\ldots$}
\put(106,62 ){\vector(-1,0){20}}
\put(110,60){$\scriptstyle{w_1\tp v^{n+1}}$}
\put(176,62 ){\vector(-1,0){20}}
\put(182,60){$\scriptstyle{w_1\tp v^{n-1}}$}
\put(229,62 ){\vector(-1,0){20}}
\put(231,60){$\scriptstyle{w_1\tp v^{n-2}}$}

\put(290,62 ){\vector(-1,0){20}}
\put(293,62 ){$\ldots$}
\put(328,62 ){\vector(-1,0){20}}
\put(340,60){$\scriptstyle{w_1\tp v^1}$}

\put(65,23){$\cdot$}\put(75,18){$\cdot$}\put(85,13){$\cdot$}

\put(355,58){\vector(0,-1){18}}
\put(354,27){$\vdots$}
\put(355,24){\vector(0,-1){18}}

\put(245,58){\vector(0,-1){18}}
\put(244,27){$\vdots$}
\put(245,24){\vector(0,-1){18}}
\put(195,58){\vector(0,-1){18}}
\put(194,27){$\vdots$}
\put(195,24){\vector(0,-1){18}}
\put(124,58){\vector(0,-1){18}}
\put(123,27){$\vdots$}
\put(124,24){\vector(0,-1){18}}

\put(106,62 ){\vector(-1,0){20}}

\put(100,0){$\scriptstyle{w_{n-1}\tp v^{n+1}}$}
\put(168,2 ){\vector(-1,0){15}}
\put(171,0){$\scriptstyle{w_{n-1}\tp v^{n-1}}$}
\put(221,1 ){\vector(-1,0){15}}
\put(222,0){$\scriptstyle{w_{n-1}\tp v^{n-2}}$}

\put(290,2 ){\vector(-1,0){20}}
\put(293,2 ){$\ldots$}
\put(328,2 ){\vector(-1,0){20}}

\put(330,0){$\scriptstyle{w_{n-1}\tp v^1}$}

\put(360,50){$\scriptstyle{f_{\al_1}}$}
\put(360,15){$\scriptstyle{f_{\al_{n-2}}}$}

\put(53,70){$\scriptstyle{f_{\al_{n}}}$}\put(93,70){$\scriptstyle{f_{\al_{2}}}$}\put(163,70){$\scriptstyle{f_{\al_{1}}}$}
\put(213,70){$\scriptstyle{f_{\al_{n-1}}}$}\put(273,70){$\scriptstyle{f_{\al_{n-2}}}$}\put(313,70){$\scriptstyle{f_{\al_2}}$}

\end{picture}
\end{center}
Here we display only the part which is relevant to our study.
In what follows, we have to  mind the arrows $f_{\al_{n-1}},f_{\al_{n}}\in \mathrm{Arr}(w_{n-1})$,
which are directed from the bottom line. Note that the rightmost square of $(n-1)\times (n-1)$ nodes is the same in
$D^n_s$ for all $s$. It is also a sub-graph in $D_\id$.

 Denote by $D_\id'\subset D_\id$  the sub-graph above the principal diagonal, i.e. $\{w_k\tp v^j\}$ such that
 $k+j\leqslant 2n-1$.
 For $s\not =\id $, let $i\in [1,n-2]$ the maximal integer displaced by $s$.
 We denote by $D_s'\subset D_s$  the trapezoid  rested on line $i$, i.e.
the set of nodes  $\{w_k\tp v^j\}$ such that $k+j\leqslant 2n-1$ and $k\leqslant i$.

\begin{lemma}
Suppose that $q^{4-2p}\not =-1$. Then $D_s'$ and ${D_s^n}'$ lie in $\hat M$.
\label{D sigma=id}
\end{lemma}
\begin{proof}
The situation is slightly different from the settings of Lemma \ref{triangle} (odd $N$), as the diagrams $D_s$ are not simple.
Applying similar arguments as in the proof of Lemma \ref{sigma=id} we check that the trapezoid in $D'_\id$ on the right of column $w_i\tp v^{n-2}$ inclusive lies in $\hat M$. The generator $f_{\al_{n-1}}$ sends the nodes of this column one step to the left modulo
maybe
one step down. Since the node $w_{n}\tp v^{n-2}$ is sent strictly leftward, column $n-2$ of $D'_\id$ is mapped onto
column $n-1$, modulo its column $n-2$, which is proved to be in $\hat M$. Therefore, column $n-1$ of $D'_\id$
lies in $\hat M$.
The bottom node of column $n$ of $D'_\id$ is $w_{n-1}\tp v^n$. Modulo $w_{n+1}\tp v^{n-2}\in \hat M$, it is the $f_{\al_n}$-image of $w_{n-1}\tp v^{n-2}\in \hat M$. Hence $w_{n-1}\tp v^n\in \hat M$. The nodes higher in
this column are also obtained from column $n-2$ via $f_{\al_n}$, which now acts strictly leftward.
Therefore, the right part of $D'_\id$ lies in $\hat M$ up to column $n$. The remaining part of $D'_\id$ to the left of
column $n$
inclusive is a triangle in a simple rectangle (just ignore the leftmost $f_{\al_n}$) and falls into Lemma \ref{triangle}.

Now suppose that $s\not=\id$ and  and let $i$ be the highest integer displaced by $s$. Contrary to $s=\id$, this case
is pretty similar to  Proposition \ref{dir_sum_symB}
for odd $N$. Notice that  right rectangular part of $D_s'$ up to column $2n-i-1$ is the same as in  $D'_\id\subset D_\id$
and lies in $\hat M$ as argued.
Since $f_{\al_{s(i)}}\not =f_{\al_i}$,  column $2n-i$ of $D'_\id$
lies in $\hat M$, by Corollary \ref{interval}.
 The remaining part of $D'_\id$ is the triangle bounded by column $2n-i$ on the right. It belongs to $\hat M$ by
Lemma \ref{triangle}.

The proof for ${D^n_s}'$ for $s=\id, (1,2)$ is similar to $D_s'$ with $s\not =\id$.  The key observation is that right rectangular part up to column $n-1$
is a sub-graph in $D'_\id$ and hence lie in $\hat M$. Further arguments are based on Lemma \ref{interval} applied
to column $n-1$ of $D'_s$.
\end{proof}

Now we are ready to prove the main result of this section.
\begin{propn}
\label{dir_sum_sym}
Suppose that $q^{2m-2p}\not =-1$. Then the tensor product $\C^N\tp M_\la$ splits to the direct sum $M_1\oplus M_2$.
\end{propn}
\begin{proof}
Similar argument as in the proof of Proposition \ref{dir_sum_symB} tells us that  $M_1\cap M_2=\{0\}$.
Indeed, the eigenvalues of $\Q$ on $\hat M_1$ and $\hat M_2$ are $\mu_1=-q^{2p}$, $\mu_2=q^{-2m}$, cf. (\ref{-1+1-1}).
They are distinct by the hypothesis, hence the sum $M_1+M_2$ is direct.
As for $\g=\s\o(2n+1)$, we need to show that $M=M_1\oplus M_2$ exhausts all of $\C^N\tp M_\la$, and it is sufficient to
prove the inclusion $\C^N\tp v_\la\subset M$.

First we consider the case $m=2$. Then $w_{2n-1}$ is the highest weight vector of the $\l$-submodule $\C^m\subset \C^N$ of
highest weight $-\ve_2$. As before, we intend to reach $w_{2n-1}\tp v_\la$ from $w_{1}\tp v_\la$ through all
$w_i\tp v_\la$ in between staying within $M$. Then we get $w_{2n}\tp v_\la\in M$ by applying $f_{\al_1}$ to $w_{2n-1}\tp v_\la$.

By Lemma \ref{D sigma=id}, the diagonal of $D_\id$ over the principal diagonal lies in $\hat M$, as it
belongs to $D_\id'$.
The vertical and horizontal arrows applied to this diagonal coincide. The same is true regarding
the nodes $w_{n-1}\tp v^{n-1}$ and $w_{n}\tp v^{n-2}$. Therefore, up to a non-zero scalar
factor,  the elements in the principal diagonal are all equivalent modulo $\hat M$. In particular,
$w_1\tp v^{2n-1} \simeq w_{2n-1}\tp v_\la \mod \hat M$.
Now notice that $x_{2}=v^{2n-1}-av^{2n-1}_s$, where $s$ is the transposition $1\leftrightarrow2$.
Observe that all other $x_{i}$, $i<n-1$ are linear combinations
of the monomials $v^{2n-1}_s$ for certain $s \not=\id$. Since $v^{2n-1}_s \in \hat M$ for such $s$ by Lemma \ref{D sigma=id},
all $x_{i}$ with $i<n-1$ belong to $\hat M$. The vector $x_{n}$ is a combination of two monomials
associated with $D_s^n$, $s=\id$, $s=(1,2)$. In view of Lemma \ref{D sigma=id} we conclude that
$w_1\tp x_{n}\in \hat M$.
Due to the symmetry between $x_{n-1}$ and $x_{n}$, we conclude that  $w_1\tp x_{n-1}\in \hat M$ too.

Since the singular vector $v_{\la-\dt}$ is a linear combination of $x_{i}$, $i=2,\ldots, n$,  the vector $w_{2n-1}\tp v_\la$
is equivalent to $w_1\tp v_{\la-\dt}$  modulo $\hat M$. Hence $w_{2n-1}\tp v_\la \in M$ as required.

Now we lift the restriction $m=2$. This is done similarly to the $\g=\s\o(2n+1)$-case. Applying the $f_{\al_1},\ldots, f_{\al_{m-1}}\in \l_-$
we get $w_i\tp v_\la\in M$ for $i=1,\ldots, m$. Lemma  \ref{gl-sym} facilitates transition to
from $w_{m}\tp v_\la\in M$ to $w_{m+1}\tp v_\la\in M$, under the assumption $q^{2m-2p}\not =-1$. Further we employ to the
quantum subgroup $U_q(\g')$, $\g'=g_{p+2}$, and reduce consideration to the case $m=2$.
This gives $w_{i}\tp v_\la\in M$, $i=m+2,\ldots N-m+2$. Finally, applying the generators
$f_{\al_1},\ldots, f_{\al_{m-1}}\in \l_-$, we descend from $w_{N-m+2}\tp v_\la\in M$ to $w_{N}\tp v_\la\in M$. This completes the proof.
\end{proof}

\begin{remark}\rm
The assumption $q^{2m- P}\not =-1$ from Propositions \ref{dir_sum_symB} and \ref{dir_sum_sym} can be
regarded as a condition on $q$ if one considers the quantum group $U_q(\g)$ over the complex field with  $q\in \C$, or over the field
$\C(q)$ of rational functions of $q$. This condition is fulfilled for
an open set including $q=1$ and therefore over the formal series in $\hbar$ with $q=e^\hbar$. Observe that
$q^{2m- P} =-1$ if and only if the eigenvalues of $\Q$ coincide, cf. (\ref{-1+1-1}). This is accountable by Lemma \ref{gl-sym},
because for such $q$ we get the inclusion $\hat M_2\subset \hat M_1$.
\end{remark}

\subsection{The module $\C^N\tp M_\la$, general $\k$}
The symmetric case worked out in detail in the preceding sections will serve as an illustration to the case
of general $\k$ considered below. However, our strategy will be slightly different, in order to save
the effort of calculating singular vectors in $\C^N\tp M_\la$. We pay a price for that by getting a weaker
result about the structure of $\C^N\tp M_\la$. Namely, instead of direct sum decomposition of
$\C^N\tp M_\la$ we construct a filtration by highest weight modules. Still it is sufficient for our purposes,
as all we need to know is the spectrum of the quantum coordinate matrix $\Q$, cf. (\ref{Q-matrix}).
 Under certain conditions, it can be extracted from
the graded module associated with filtration as well as from direct sum decomposition.

We have the irreducible decomposition
$$
\C^N=\C^{n_1}\oplus\ldots \oplus  \C^{n_\ell}\oplus \C^{m}\oplus \C^{P}\oplus \C^{m}\oplus \C^{n_\ell}\oplus \ldots \oplus \C^{n_1}
$$
of the natural $\g$-representation $\C^N$ into $\l$-blocks. The submodules $\C^{n_i}$ carry the natural
 and conatural representations the block $\g\l(n_i)\subset \l$, $i=1,\ldots,\ell+1$, with $n_{\ell+1}=m$.
The submodule $\C^{P}$ supports the natural representation of the block $\s\o(P)\subset\l$.
We enumerate these submodules from left to right
 as $W_i$, $i=1,\ldots, 2\ell+3$.
This decomposition is
compatible with the standard basis $\{w_i\}$, and the basis element with the lowest number falling into the block
is its highest weight vector.
Let $\nu_i$, $i=1,\ldots, 2\ell+3$, be the highest weights of the irreducible blocks and let  $w_{\nu_i}\in W_i$ denote their
highest weight vectors. As we said, they form a subset of the standard basis.
Explicitly, the highest weights of the blocks are
$\nu_i=\ve_{n_1+\ldots +n_{i-1}+1}$ for $i=1,\ldots,\ell+2$ and
$\nu_{2\ell+4-i}=-\ve_{n_1+\ldots +n_{i}}$ for $i=1,\ldots,\ell+1$.

For generic $\la$ this decomposition gives rise to the decomposition
$
\C^N\tp \hat M_\la=\oplus_{i=1}^{2\ell+3}\hat M_i
$
where each $\hat M_i$ is the parabolic Verma module induced from $W_i\tp \C_\la$.
Let $M_i$ denote its image under the projection to $\C^N\tp M_\la$.

Transition to the isotropy subalgebra $\k\supset \l$ merges two copies of $\C^m$  up into
a single irreducible $\k$-submodule. As a result, $M_{\ell+3}$ should disappear from
$\C^N\tp M_\la$. We saw this effect for $\ell=0$ and we expect it for general $\k$. However, constructing the direct sum decomposition
of $\C^N\tp M_\la$ along the same lines requires the knowledge of singular vectors for all $\hat M_i$.
Instead, we work with a filtration, which construction is much easier. We do not even
check that each graded component, apart from the ${\ell+3}$-d, survives in the projection  $\C^N\tp \hat M_\la\to \C^N\tp M_\la$.
We just need to make sure that the elementary divisor corresponding to the quotient $V_{\ell+3}/V_{\ell+2}$ drops from the minimal polynomial
of $\Q$.

For all $j=1,\ldots, 2\ell+3$ we denote by $ \hat V_j$ the submodule in  $ \C^{N}\tp  \hat M_\la$ generated by
$\{w_{\nu_i}\tp v_\la\}_{i=1,\ldots, j}$. Let $ V_j$ denote its image in  $ \C^{N}\tp  M_\la$.
We have the obvious inclusions $\hat V_{j-1}\subset \hat V_{j}$, $V_{j-1}\subset  V_{j}$.
It is convenient to set $\hat V_0$ and  $V_0$ to $\{0\}$.

Propositions \ref{dir_sum_symB} and  \ref{dir_sum_sym}, which are formulated for the symmetric case, can be restated in a milder setting
as $\C^N\tp M_\la\simeq V_1\oplus V_2/V_1$. The equality $\C^N\tp M_\la=V_2=V_3$ remains true if we relax the assumption on $\la$. This assumption facilitates the equivalence $u_{\nu_2}\simeq w_{m+1}\tp v_\la\mod V_1$, so the proof
remains essentially the same if $u_{\nu_2}$ is replaced with $w_{m+1}\tp v_\la$. Here we establish a generalization of this fact.

\begin{propn}
\label{filt_WM}
The submodules $\{0\}=V_0\subset V_1\subset \ldots \subset V_{2\ell+3}$ form an ascending filtration of
$ \C^{N}\tp  M_\la $.
For each $k=1,\ldots, 2\ell+3$, the graded component $V_{k}/V_{k-1}$ is either $\{0\}$ or
 generated by (the image of) $w_{\nu_k}\tp v_\la$,
which is  the highest weight vector in $V_{k}/V_{k-1}$.
In particular, $V_{\ell+2}=V_{\ell+3}$.
\end{propn}
\begin{proof}
Our strategy is similar to the proof of Propositions \ref{dir_sum_symB} and  \ref{dir_sum_sym}.
We mean to  show that  $\oplus_{i=1}^k W_i\tp v_\la\subset V_k$, which in particular imples $\C^{N}\tp v_\la\subset
V_{2\ell+3}$ for $k=2\ell+3$.
Then $e_{\al}(w_{\nu_k}\tp v_\la)=0 \!\!\mod V_{k-1}$, i.e. $w_{\nu_k}\tp v_\la$ is a singular vector in $V_{k-1}/V_{k}$
if not zero.
Since $V_{k-1}/V_{k}$ is generated by $w_{\nu_k}\tp v_\la$, it is then the highest weight vector.
  This will imply $\C^{N}\tp v_\la\subset V_{2\ell+3}$ and
$V_{2\ell+3}=   M_\la$.

Thus, we wish to prove that  $W_k\tp v_\la\subset V_k$. This is true
for $k=0$ if we set $W_0=\{0\}$. Suppose we have done this
for some $k\geqslant 0$.
By construction, $w_{\nu_{k+1}}\tp v_\la\in V_{k+1}$. Consecutively applying
the Chevalley generators from the  block of $\l_-$ which does not vanish on $W_{k+1}$ we conclude
that $W_{k+1}\tp v_\la\subset V_{k+1}$. Induction on $k$ proves $W_k\tp v_\la\subset V_k$
for all $k$.

Finally, the equality $V_{l+2}= V_{l+3}$ follows from the inclusion $W_{l+3}\tp v_\la\subset V_{l+2}$,
and this boils down to the symmetric case.
Indeed, let $\g'=\g_{p+2}\subset \g$ be as defined in Section \ref{module_M_general_k}.
Let $\hat M_\la'\subset \hat M_\la$ be its parabolic Verma submodule generated by $v_\la$ and
let $V_{\ell+2}'$ be the $U_q(\g')$-submodule generated by $w_{\mu_{\ell+1}}\tp v_\la,w_{\mu_{\ell+2}}\tp v_\la$.
As we discussed in the symmetric case, $W_{\ell+3}\tp v_\la\subset V_{\ell+2}'\subset V_{\ell+2}$.
Hence $V_{l+3}=V_{l+2}$, and the proof is complete.
\end{proof}
\section{The matrix of quantum coordinate functions}
\label{secQCC}
Similarly to classical conjugacy classes, their quantum counterparts are described
through a matrix  $A$ of non-commutative "coordinate functions"
or its image $\Q\in  \End(\C^{N})\tp U_q(\g)$, which should be regarded as "restriction"
 of $A$ to the  "quantum group $G_q"$.  In this section we study algebraic properties of $\Q$.

The operator $\Q$ is defined  through the universal  R-matrix $\Ru$,
which is an invertible element of (completed) tensor square of  $U_\hbar(\g)$:
\be
\Q=(\pi\tp\id)(\Ru_{12}\Ru)\in \End(\C^{N})\tp U_q(\g).
\label{Q-matrix}
\ee
Here $\pi$ is the representation homomorphism $U_\hbar(\g)\to \End(\C^{N})$. The matrix $\Q$ commutes with
$(\pi\tp \id)\circ\Delta(u)$ for all $u\in U_q(\g)$
 producing an invariant operator on $\C^N\tp V$ for every $U_q(\g)$-module $V$.

Let $\rho$ denote the half-sum of all positive roots $\rho=\frac{1}{2}\sum_{\al\in \Rm_+}\al$.
In the orthogonal basis of weights $\{\ve_i\}$, it reads
$$
\rho=\sum_{i=1}^n \rho_i\ve_i,
\quad
\rho_i=\rho_1-(i-1),
\quad
\rho_1=
\left\{
\begin{array}{rcccl}
n-\frac{1}{2}&\mbox{for}& \g&=&\s\o(2n+1),\\
n-1&\mbox{for}& \g&=&\s\o(2n).\\
\end{array}
\right.
$$
Regarded as an operator on $\C^{N}\tp \hat M_\la$, the element $\Q$ satisfies a polynomial equation with the roots
$$
q^{2(\la+\rho,\nu_i)-2(\rho,\ve_1)+(\nu_i,\nu_i)-1}=
\left\{
\begin{array}{rccccl}
q^{2(\la,\nu_i)+2(\rho,\nu_i-\ve_1)}&\mbox{for}&p>0,&\\
q^{-2n}&\mbox{for}&p=0,&i=\ell+2,& \g=\s\o(2n+1).\\
\end{array}
\right.
.
$$
where $\nu_i$, $i=1,\ldots, 2\ell+3$  are the highest weights of the irreducible $\l$-submodules in $\C^{N}$, \cite{M2}, Theorem 4.2.
The bottom line corresponds to zero $\nu_i$, which is present only for odd $N$ if $p=0$.

Assuming $\la \in \mathfrak{C}_{\l,reg}^*$, put $\La_i=(\la,\ve_{n_1+\ldots+n_{i-1}+1})=(\la,\ve_{n_1+\ldots+n_{i}})$
for $i=1,\ldots,\ell+2$
(recall that $n_{\ell+1}=m$ and $n_{\ell+2}=p$, by our convention). The weight
$\la$ depends on the parameters $(\La_i)$, with $\La_{\ell+2}=0$.
Define the vector $\mub$ by
\be
\mu_i=q^{2\La_i-2({n_1+\ldots+n_{i-1}})}, \quad i=1,\ldots, \ell+2.
\label{mu_param}
\ee
The eigenvalues of $\Q$ on $\End(\C^{N}\tp \hat M_\la)$  are expressed through $\mub$ by
\be
\label{e_v in hatM_la}
\mu_i,\quad \mu_i^{-1}q^{-4\rho_1+2(n_i-1)}=\mu_i^{-1}q^{-2N+2(n_i+1)},\quad i=1,\ldots,\ell+1,\quad\mu_{\ell+2}.
\ee
As was mentioned the operator $\Q$ on $\C^N\tp \hat M_\la$ satisfies a polynomial equation of degree $2\ell+3$.
Formula (\ref{e_v in hatM_la}) implies that, at generic point $\la\in \mathfrak{C}_{\l,reg}^*$, the roots of the polynomial
are pairwise distinct for almost all $q$. Hence $\Q$ is semisimple for almost all $q$ at generic $\la$.
In particular, the eigenvalues $\mu_{\ell+1}$, $\mu_{\ell+2}$, and $\mu_{\ell+3}$ read
\be
\begin{array}{llllllllllllll}
 \mu_{\ell+1}&=&-q^{-2p}, &\mu_{\ell+2}&=&q^{-2m},&  \mu_{\ell+3}&=&-q^{-2n+2},&N=2n,\\
 \mu_{\ell+1}&=&-q^{-2p-1},& \mu_{\ell+2}&=&q^{-2m},&  \mu_{\ell+3}&=&-q^{-2n+1},&N=2n+1.
 \end{array}
 \label{-1+1-1}
\ee
Note that $\mu_{\ell+1}$ may be equal to $\mu_{\ell+3}$ only for $m=1$, which case is excluded from our consideration. In other words, the  minimal polynomial of $\Q$ remains semisimple on $\C^N\tp \hat M_\la$
for almost all $q$
upon specialization of $\la$ to  generic point of $\mathfrak{C}_{\k,reg}^*$. Therefore $\Q$ is semisimple on $\C^N\tp \hat M_\la$
and hence on $\C^N\tp  M_\la$ for an open set in $\mathfrak{C}_{\k,reg}^*$, for almost all $q$.
%\be
% \mu_{\ell+1}=-q^{\pm (2p+1)-2(n-p-m)}, \quad \mu_{\ell+2}=q^{-2(n-p)}, \quad  \mu_{\ell+3}=-q^{\pm (2p+1)-2(n-p-m)-2N+2(m+1)}
%\nn
%\ee

The vector $\mub$ belongs to $\hat \Mc_K$ modulo $\hbar$ for $\la\in \mathfrak{C}_{\k,reg}^*$. Recall that
$\hat \Mc_K$ parameterizes the moduli space $\Mc_K$ of conjugacy classes with given $K$.
\begin{propn}
\label{propn_evQM}
For $\la\in \mathfrak{C}_{\k,reg}^*$ the operator  $\Q$  satisfies a
polynomial equation of degree $2\ell+2$ on $\C^{N}\tp M_\la$ with the roots
\be
\mu_i,\quad \mu_i^{-1}q^{-2N+2(n_i+1)},\quad i=1,\ldots,\ell,\quad\mu_{\ell+1},\quad\mu_{\ell+2}.
\label{e_v in M_la}
\ee
\end{propn}
\begin{proof}
For generic $q$, the operator $\Q\in \End(\C\tp \hat M_\la)$ is semisimple, and
the roots (\ref{e_v in hatM_la}) are pairwise distinct. Therefore,
the projection of $\Q$ to $\End(\C\tp M_\la)$ is semisimple for almost all $q$ and satisfies the same polynomial equation. Since the graded components $V_{k}/V_{k-1}$ are highest weight modules,
the projection of $\Q$ is scalar on each $V_{k}/V_{k-1}$, which is one of the eigenvalues of $\Q$.  By Proposition \ref{filt_WM},
the eigenvalue $\mu_{\ell+3}$ drops from the spectrum of $\Q$ on $\C\tp M_\la$,
hence the simple
divisor $\Q-\mu_{\ell+3}$ is invertible and can be canceled from the polynomial.
\end{proof}
%It can be shown that the polynomial from Proposition \ref{propn_evQM} is minimal
%for generic $\la\in \mathfrak{C}^*_{\k,reg}$ and $q\in \C$
%and cannot be reduced further.
%Equivalently, only $V_{\ell+3}/V_{\ell+2}$ vanishes from the sum $\oplus_{i=1}^{2\ell+3}V_{i}/V_{i-1}$. We do not focus on this issue %here.

The matrix $\Q$ produces the center of $U_q(\g)$ via the q-trace construction.
For any invariant matrix $X\in \End(\C^{N})\tp \A$
with the entries in a $U_q (\g)$-module $\A$, one can define an invariant element
\be
\label{q-trace}
\Tr_{q}(X):=\Tr\bigl(q^{2h_\rho}X\bigr)\in \A.
\ee
Recall that $h_\rho$ is an element from $\h$ such that $\al(h_\rho)=(\al,\rho)$ for all $\al\in \h^*$.
The $q$-trace, when applied $X=\Q^k$, $k\in \Z_+$, gives
a series of central elements of $U_q (\g)$. We will use the shortcut notation $\tau_k=\Tr_q({\Q^k})$.

A module $M$ of highest weight $\la$ defines a one dimensional representation $\chi^\la$ of the center of  $U_q (\g)$,
which assigns a scalar to each  $\tau_\ell$, \cite{M2}, formula (24):
\be
\chi^\la(\tau_k)=
\sum_{\nu} q^{2k(\la+\rho,\nu)-2k(\rho,\ve_1)+k(\nu,\nu)-k}
\prod_{\al\in \Rm_+}\frac{ q^{(\la+\nu+\rho,\al)}-q^{-(\la+\nu+\rho,\al)}}{ q^{(\la+\rho,\al)}-q^{-(\la+\rho,\al)}}.
\label{char_V}
\ee
The summation is taken over  weights $\nu$ of the module $\C^{N}$. The term
$k(\nu,\nu)-k$ survives  for $\nu=0$, which is the case only for odd $N$.
Restriction of $\la$ to $\mathfrak{C}_{\k,reg}^*$ makes
the right hand side a function of $\mub$ defined in (\ref{mu_param}). We denote this function by
$\vt_{\nb,q}^k(\mub)$, where $\nb=(n_1,\ldots,n_\ell, m,p)$ is the integer valued vector of multiplicities.
In the limit $\hbar \to 0$ the function $\vt_{\nb,q}^k(\mub)$ goes over
into the right hand side of (\ref{tr_cl}), where
$\mu_i=\lim_{h\to 0}q^{2(\la,\nu_i)}$, $i=1,\ldots,\ell$.

\section{Quantum conjugacy classes of non-Levi type}
By quantization of a commutative $\C$-algebra $\A$ we understand a
$\C[\![\hbar]\!]$-algebra $\A_\hbar$, which is free as a $\C[\![\hbar]\!]$-module
and $\A_\hbar/\hbar\A_\hbar\simeq \A$ as a $\C$-algebra. Quantization is called $U_\hbar(\g)$-equivariant if
$\A$ and $\A_\hbar$ are, respectively $U(\g)$- and $U_\hbar(\g)$-module algebras and the $U_\hbar(\g)$-action
on $\A_\hbar$ is a deformation of the $U(\g)$-action on $\A$.
Below we describe
the quantization of $\C[G]$ along the Poisson bracket (\ref{poisson_br_sts}).

Recall from \cite{J} that the image of the universal R-matrix of the quantum group $U_\hbar(\g)$ in the
defining representation is equal, up to
a scalar factor, to
$$
R=\sum_{i,j=1 }^{N} q^{\delta_{ij}-\delta_{ij'}}e_{ii}\tp e_{jj}
  +
  (q-q^{-1})\sum_{i,j=1 \atop i>j}^{N}(e_{ij}\tp e_{ji}
- q^{\rho_i-\rho_j}
e_{ij}\tp e_{i'j'}).
$$
The coefficients $\rho_i$ are defined as
$\rho_{n+1}=0$, $\rho_i=-\rho_{i'}=(\rho,\ve_i)=n+\frac{1}{2}-i$ for $N=2n+1$ and
$\rho_i=-\rho_{i'}=(\rho,\ve_i)=n-i$ for $N=2n$, where $i$ runs over $1,\ldots, n$.

Denote by $S$ the $U_\hbar(\g)$-invariant operator $PR\in \End(\C^{N})\tp \End(\C^{N})$,
where $P$ is the ordinary flip of $\C^{N}\tp \C^{N}$.
This matrix has three invariant projectors to its eigenspaces,
among which there is a one-dimensional projector $\kappa$ to the
 trivial  $U_\hbar(\g)$-submodule, proportional to
$\sum_{i,j=1}^{N}q^{\rho_i-\rho_j} e_{i'j}\tp e_{ij'}.
$

Denote by  $\C_\hbar[G]$ the associative algebra generated by
the entries of a matrix $A=(A_{ij})_{i,j=1}^{N}\in \End(\C^{N})\tp \C_\hbar[G]$
modulo the relations
\be
S_{12}A_2S_{12}A_2=A_2S_{12}A_2S_{12}
,\quad A_2S_{12}A_2\kappa=q^{-{N}+1}\kappa=\kappa A_2S_{12}A_2.
\label{A-matrix}
\ee
These relations are understood in $\End(\C^{N})\tp \End(\C^{N})\tp \C_\hbar[G]$,
and the indices distinguish the two copies of $\End(\C^{N})$, in the usual
way.

The algebra $\C_\hbar[G]$ is a quantization of $\C[G]$ along the Poisson bracket
(\ref{poisson_br_sts}).
 It carries a $U_\hbar(\g)$-action, which is  a deformation of the conjugation action of
$U(\g)$ on $\C[G]$. This action is determined by the requirements that
$A$ commutes with $(\pi\tp \id)\circ\Delta U_\hbar(\g)$ in the tensor product $\End(\C^{N})\tp \C_\hbar[G]\rtimes U_\hbar(\g)$,
where $\pi\colon   U_\hbar(\g)\to \End(\C^{N})$ is the representation homomorphism.
It is important that $\C_\hbar[G]$ can be realized as a $U_\hbar(\g)$-invariant subalgebra in $U_q(\g)$,
with respect to the adjoint action.
The embedding is implemented via the assignment
$$
 \End(\C^{N})\tp \C_\hbar[G]\ni
 A\mapsto \Q\in \End(\C^{N})\tp  U_q(\g).
$$

The following properties of $\C_\hbar[G]$  will be of importance.
Denote by $I_\hbar(G)\subset \C_\hbar[G]$ the subalgebra of $U_\hbar(\g)$-invariants, which
also coincides with the center of $\C_\hbar[G]$. For $N=2n+1$ it is generated by the
q-traces $\Tr_q(\A^l)$, $l=1,\ldots, N$.
Not all traces are independent, but that is immaterial for this consideration.
Traces of $A^l$ are not enough for $N=2n$,
and one should add one more invariant $\tau^-$ in order to get entire $I_\hbar(G)$. On a module of highest
weight $\la$, this invariant returns $\chi^\la(\tau^-)=\prod^n_{i=1}(q^{2(\la+\rho,\ve_i)}-q^{-2(\la+\rho,\ve_i)})$, see
Proposition 7.4, \cite{M2}. However, it vanishes on modules with highest weight $\la \in \mathfrak{C}_{\l}^*$, so we take no care of it.

\begin{thm}
\label{QCC}
Suppose that $\la=\mathfrak{C}_{\k,reg}^*$ is admissible, and let $\mub$ be as in (\ref{mu_param}).
The quotient of $\C_\hbar[G]$ by the ideal of relations
\be
\prod_{i=1}^{\ell}(\Q-\mu_i)\times (\Q-\mu_{\ell+1})(\Q-\mu_{\ell+2})\times
\prod_{i=1}^{\ell}(\Q-\mu_i^{-1}q^{-2N+2(n_i+1)})=0,
\label{q-min_pol}
\ee
\be
\Tr_q(\Q^k)=\vt_{\nb,q}^k(\mub)
\label{q-traces}
\ee
is an equivariant quantization of the class $\lim_{\hbar\to 0}\mub=\mub^0\in\hat \Mc_K$.
It is the image of $\C_\hbar[G]$ in the algebra of endomorphisms
of the $U_q(\g)$-module $M_\la$.
\end{thm}
\begin{proof}
The proof is similar to \cite{M5}, Theorem 10.1. and \cite{M2}, Theorem 8.2., and we give its sketch  here. It is based on equivariant homomorphism from
$\C_\hbar[G]$ to $\End(M_\la)\subset \End(M_\la)\tp \C(\!(\hbar)\!)$, where the extension of $M_\la$  by Laurent series in $\hbar$ is taken to enable the $U_\hbar(\g)$-action. While
$M_\la$ is only a $U_q(\g)$- but not a $U_\hbar(\g)$-module, $\End(M_\la)$ is $U_\hbar(\g)$-invariant as well as the image of $\C_\hbar[G]$ in it, see \cite{M5} for details.
This homomorphism factors through a homomorphism $\Psi\colon \C_\hbar[G]/J_\la\to \End(M_\lambda)$, where $J_\la$ is the ideal generated by the kernel
of the central character $\chi^\la$; it  is defined by the relations (\ref{q-traces}).
The $U_\hbar(\g)$-algebra $\C_\hbar[G]/J_\la$ is a direct sum of isotypical components
of finite rank over $\C[\![\hbar]\!]$, which follows from  \cite{M1},
Theorem 5.4. Therefore, the image of $\Psi$
is  free over $\C[\![\hbar]\!]$. It can be shown that the algebra $\C_\hbar[G]/J_\la$ is free over $\C[\![\hbar]\!]$,
hence $(\ker \Psi)_0=\ker \Psi/ \hbar(\ker \Psi) $ is isomorphically embedded in $\C_\hbar[G]/J_\la\mod \hbar$.
The kernel $\ker \Psi$ contains the ideal $J$ generated by the entries of the matrix
polynomial in the left-hand-side of (\ref{q-min_pol}). In the classical limit, $J$ goes over to the defining ideal
of the conjugacy class, by Theorem \ref{prop_clas_so}. The latter is a maximal proper invariant ideal
and hence  equal to $(\ker \Psi)_0$. Therefore, the embedding  $J\subset \ker \Psi$ is an isomorphism by the Nakayama lemma.
\end{proof}
Theorem \ref{QCC} describes quantization in terms of the matrix $\Q$, which is the image of the matrix $A$.
To obtain the description in terms of $A$, one should replace $\Q$ with $A$ in (\ref{q-min_pol}) and (\ref{q-traces})
and add the relations  (\ref{A-matrix}).

The quantization we have constructed is equivariant with respect to the standard or Drinfeld-Jimbo quantum group
$U_\hbar(\g)$. Other quantum groups are obtained from standard $U_\hbar(\g)$ by twist, \cite{ESS}.
Formulas (\ref{q-min_pol}) and (\ref{q-traces}) are valid for
any quantum group  $U_\hbar(\g)$ upon the following modifications. The matrix $\Q$ is expressed through the universal R-matrix as
usual. The q-trace should be redefined
as $\Tr_q(X)=q^{1+2(\rho,\ve_1)}\Tr\Bigl(\pi\bigl(\gm^{-1}(\Ru_1)\Ru_2\bigr)X\Bigr)=q^{N-1}\Tr\Bigl(\pi\bigl(\gm^{-1}(\Ru_1)\Ru_2\bigr)X\Bigr)$.
This can be verified along the lines of \cite{MO}.

\vspace{0.5in}
{\bf Acknowledgements}.
The author is extremely grateful to the Max-Planck Institute for Mathematics in Bonn, where this work has been done, for warm hospitality and
excellent research atmosphere.
This study is also supported in part by the RFBR grant 09-01-00504.

\end{document}